\documentclass[a4paper,oneside,10pt]{article}%
\usepackage{amsmath}
\usepackage{amsfonts}
\usepackage{amssymb}
\usepackage{graphicx}
\usepackage{color}
\usepackage[square,numbers,sort&compress]{natbib}%
\providecommand{\U}[1]{\protect \rule{.1in}{.1in}}

\newtheorem{theorem}{Theorem}[section]

\newtheorem{corollary}[theorem]{Corollary}

\newtheorem{example}[theorem]{Example}

\newtheorem{proposition}[theorem]{Proposition}
\newtheorem{remark}[theorem]{Remark}

\newenvironment{proof}[1][Proof]{\noindent \textbf{#1.} }{\  $\Box$}
\numberwithin{equation}{section}

\begin{document}

\title{Relationship between stochastic maximum principle and dynamic programming
principle under convex expectation}
\author{Xiaojuan Li\thanks{Department of Mathematics, Qilu Normal University, Jinan
250200, China. lxj110055@126.com. }
\and Mingshang Hu \thanks{Zhongtai Securities Institute for Financial Studies,
Shandong University, Jinan, Shandong 250100, China. humingshang@sdu.edu.cn.
Research supported by NSF (No. 12326603, 11671231).} }
\maketitle

\textbf{Abstract}. In this paper, we study the relationship between maximum
principle (MP) and dynamic programming principle (DPP) for forward-backward
control system under consistent convex expectation dominated by $G$%
-expectation. Under the smooth assumptions for the value function, we get the
relationship between MP and DPP under a reference probability by establishing
a useful estimate. If the value function is not smooth, then we obtain the
first-order sub-jet and super-jet of the value function at any $t$. However,
the processing method in this case is much more difficult than that when $t$
equals $0$.



\addcontentsline{toc}{section}{\hspace*{1.8em}Abstract}

\section{Introduction}

Volatility uncertainty is an important research topic in economy and finance,
which cannot be studied in a probabilistic framework (see \cite{DM, EJ2, FPW,
MPZ, Peng2004} and the references therein). Based on \cite{Peng2004}, Peng
\cite{P07a, P08a} proposed the theory of $G$-expectation (consistent sublinear
expectation) and established the theory of stochastic differential equation
driven by $G$-Brownian motion $B$ ($G$-SDE) under $G$-expectation
$\mathbb{\hat{E}}$. Denis et al. \cite{DHP11} proved that the $G$-expectation
$\mathbb{\hat{E}}$ can be represented as the supremum of a family of linear
expectations $\{E_{P}:P\in \mathcal{P}\}$, where $\mathcal{P}$ is a set of
weakly compact and nondominated probability measures. Since $\mathcal{P}$ is
nondominated, the representation of conditional $G$-expectation is very
difficult and completely different from the representation of conditional
$g$-expectation (see \cite{P1997} for the definition of $g$-expectation).
Soner et al. \cite{STZ} first introduced $\mathcal{P}(t,P)$ (see
(\ref{nn-e3-6}) for definition) for each $P\in \mathcal{P}$ and obtained the
representation theorem for conditional $G$-expectation $\mathbb{\hat{E}}_{t}$
based on $\mathcal{P}(t,P)$, which plays a key role in establishing the theory
of a new type of fully nonlinear backwad SDE, called $2$BSDE (see
\cite{STZ11}). Recently, the representation theorem for conditional convex
$\tilde{G}$-expectation $\mathbb{\tilde{E}}_{t}$, which is dominated by
$G$-expectation $\mathbb{\hat{E}}_{t}$, has been obtained in \cite{LH} based
on $\mathcal{P}(t,P)$.

Different from the method in \cite{STZ11}, Hu et al. \cite{HJPS1, HJPS}
established the theory of BSDE driven by $G$-Brownian motion ($G$-BSDE). The
reader may refer to \cite{HYH, HTW, LRT} and the references therein for the
recent developments on these two kinds of BSDEs. The maximum principle (MP)
and the dynamic programming principle (DPP) for forward-backward control
system under $G$-expectation have been obtained in \cite{HJ0} and \cite{HJ1}
respectively. Li \cite{L1} studied the relationship between MP and DPP for the
following forward-backward control system under $G$-expectation:%
\begin{equation}
dX_{t}^{u}=h(t,X_{t}^{u},u_{t})d\langle B\rangle_{t}+\sigma(t,X_{t}^{u}%
,u_{t})dB_{t},\text{ }X_{0}^{u}=x_{0}\in \mathbb{R}, \label{nn-e1-1}%
\end{equation}%
\begin{equation}
dY_{t}^{u}=-g(t,X_{t}^{u},Y_{t}^{u},Z_{t}^{u},u_{t})d\langle B\rangle
_{t}+Z_{t}^{u}dB_{t}+dK_{t}^{u},\text{ }Y_{T}^{u}=\Phi(X_{T}^{u}),
\label{nn-e1-2}%
\end{equation}
where the control domain $U\subset \mathbb{R}^{m}$ is a given nonempty convex
and compact set, the set of admissible controls is denoted by $\mathcal{U}%
[0,T]$. The cost functional is defined by $J(u):=Y_{0}^{u}$, and the optimal
control problem is%
\begin{equation}
J(u^{\ast})=\min_{u\in \mathcal{U}[0,T]}J(u). \label{nn-e1-3}%
\end{equation}
The solution of control system (\ref{nn-e1-1}) and (\ref{nn-e1-2})
corresponding to an optimal control $u^{\ast}$ is denoted by $(X^{\ast
},Y^{\ast},Z^{\ast},K^{\ast})$. By using the strict comparison theorem for
BSDEs under a reference probability $P^{\ast}$ and the weak convergence
method, Li derived $Y_{t}^{\ast}=V(t,X_{t}^{\ast})$, $P^{\ast}$-a.s., under
the smooth assumptions for the value function $V(\cdot,\cdot)$ and
$D_{x}^{1,-}V(t,x)$, $D_{x}^{1,+}V(t,x)$ in the viscosity sense at the initial
time $(t,x)=(0,x_{0})$. In particular, $Y_{t}^{\ast}=V(t,X_{t}^{\ast})$ may
not hold for other $P\in \mathcal{P}$ and $D_{x}^{1,+}V(0,x_{0})$ may be empty,
which are different from the classical results. The relationship between MP
and DPP in a probabilistic framework can be found in \cite{BCM, HJX, NSW,
J.Yong} and the references therein.

In some financial problems, we need to consider convex expectations (see
\cite{FS}). So it is necessary to study the stochastic optimal control problem
under convex expectation. The MP and DPP for forward-backward control system
under $\tilde{G}$-expectation dominated by $G$-expectation have been obtained
in \cite{LH} and \cite{MJL0} respectively. In this paper, we want to study the
relationship between MP and DPP for forward-backward control system under
$\tilde{G}$-expectation $\mathbb{\tilde{E}}$, which consists of (\ref{nn-e1-1}%
) and the following BSDE under $\tilde{G}$-expectation%
\begin{equation}
Y_{t}^{u}=\mathbb{\tilde{E}}_{t}\left[  \Phi(X_{T}^{u})+\int_{t}^{T}%
g(s,X_{s}^{u},Y_{s}^{u},u_{s})d\langle B\rangle_{s}\right]  . \label{nn-e1-4}%
\end{equation}
The optimal control problem is defined by (\ref{nn-e1-3}).

During the process of studying the above mentioned problem, we mainly
encountered two key difficult problems. One of the difficulties is that
(\ref{nn-e1-4}) is not a BSDE under a reference probability $P^{\ast}$, so we
cannot use the strict comparison theorem for BSDEs to get the relationship
between MP and DPP under the smooth assumptions for the value function
$V(\cdot,\cdot)$. Therefore, we establish a useful estimate (see
(\ref{nnn-e3-4})) based on the representation theorem for $\mathbb{\tilde{E}%
}_{t}$, and then obtain the relationship between MP and DPP under $P^{\ast}$.

If the value function $V(\cdot,\cdot)$ is not smooth, we want to study
$D_{x}^{1,-}V(t,X_{t}^{\ast})$ and $D_{x}^{1,+}V(t,X_{t}^{\ast})$ in the
viscosity sense, which has more difficulties than those we encountered when
studying $D_{x}^{1,-}V(0,x_{0})$ and $D_{x}^{1,+}V(0,x_{0})$. Because
$X_{t}^{\ast}$ is a random variable, we have to study $D_{x}^{1,-}%
V(t,X_{t}^{\ast})$ and $D_{x}^{1,+}V(t,X_{t}^{\ast})$ under a probability $P$.
From Steps 6-7 in the proof of Theorem \ref{MP-DPP-2}, we understand that we
can obtain $D_{x}^{1,-}V(t,X_{t}^{\ast})$ and $D_{x}^{1,+}V(t,X_{t}^{\ast})$
on the set $\{Y_{t}^{\ast}=V(t,X_{t}^{\ast})\}$, but we cannot obtain
$D_{x}^{1,-}V(t,X_{t}^{\ast})$ and $D_{x}^{1,+}V(t,X_{t}^{\ast})$ on the set
$\{Y_{t}^{\ast}\not =V(t,X_{t}^{\ast})\}$. So we define
\[
\mathcal{\tilde{P}}_{t}^{\ast}=\{P\in \mathcal{P}:Y_{t}^{\ast}=V(t,X_{t}^{\ast
}),\text{ }P\text{-a.s.}\}
\]
and then prove that $\mathcal{\tilde{P}}_{t}^{\ast}$ is nonempty, thus, we
study $D_{x}^{1,-}V(t,X_{t}^{\ast})$ and $D_{x}^{1,+}V(t,X_{t}^{\ast})$ under
$P\in \mathcal{\tilde{P}}_{t}^{\ast}$. However, in order to obtain $D_{x}%
^{1,-}V(t,X_{t}^{\ast})$ and $D_{x}^{1,+}V(t,X_{t}^{\ast})$ under $P$, we need
to prove the almost sure convergence of $\{p_{t}^{Q_{n}}:n\geq1\}$, where
$(p^{Q_{n}},q^{Q_{n}},N^{Q_{n}})$ is the solution of the adjoint equation
(\ref{nn-e2-5}) under $Q_{n}$ and $Q_{n}\in \mathcal{P}_{\varepsilon_{n}}%
^{\ast}(t,P)$ (see (\ref{new-nnn-e3-1}) for definition) with $\varepsilon
_{n}\downarrow0$ as $n\rightarrow \infty$. Since $\mathcal{P}$ is weakly
compact, we can only prove that $\{p_{t}^{Q_{n}}:n\geq1\}$ has a subsequence
which converges weakly in $L^{2}$ under $P$. In this way, we are unable to
acquire that $\{p_{t}^{Q_{n}}:n\geq1\}$ is almost surely convergent, which is
the second difficulty in the entire process of obtaining $D_{x}^{1,-}%
V(t,X_{t}^{\ast})$ and $D_{x}^{1,+}V(t,X_{t}^{\ast})$ under $P$. Through some
detailed estimates, we can obtain an important estimate (\ref{nnn-e3-29}).
Then we can discover an interesting conclusion: For each $Q_{n}\in
\mathcal{P}_{\varepsilon_{n}}^{\ast}(t,P)$, when $\varepsilon_{n}\downarrow0$
as $n\rightarrow \infty$, then $p_{t}^{Q_{n}}\rightarrow \bar{p}_{t}$, $P$-a.s.;
Meanwhile, for each $Q_{n}^{\prime}\in \mathcal{P}_{\varepsilon_{n}^{\prime}%
}^{\ast}(t,P)$, when $\varepsilon_{n}^{\prime}\uparrow0$ as $n\rightarrow
\infty$, then $p_{t}^{Q_{n}^{\prime}}\rightarrow p_{t}$, $P$-a.s., where
$\bar{p}_{t}$ and $p_{t}$ are defined in Theorem \ref{MP-DPP-2}. Based on this
convergence property, we obtain that $D_{x}^{1,-}V(t,X_{t}^{\ast}%
)\subseteq \lbrack p_{t},\bar{p}_{t}],$ $P$-a.s., and $\{p_{t}\} \subseteq
D_{x}^{1,+}V(t,X_{t}^{\ast})$, $P$-a.s., if $p_{t}=\bar{p}_{t}$, $P$-a.s.

This paper is organized as follows. In Section 2, we recall some basic results
of MP and DPP under convex expectation dominated by $G$-expectation. In
Section 3, we first get the relationship between MP and DPP under the smooth
assumptions for the value function $V(\cdot,\cdot)$, and then obtain
$D_{x}^{1,-}V(t,X_{t}^{\ast})$ and $D_{x}^{1,+}V(t,X_{t}^{\ast})$ under each
fixed $P\in \mathcal{\tilde{P}}_{t}^{\ast}$. The proof of (\ref{nnn-e3-4}) is
given in appendix.

\section{Preliminaries}

In this section, we recall some basic results of MP and DPP under convex
expectation dominated by $G$-expectation. The readers may refer to \cite{MJL0,
LH} for more details.

Let $T>0$ be fixed and let $\Omega_{T}=C_{0}([0,T];\mathbb{R}^{d})$ be the
space of $\mathbb{R}^{d}$-valued continuous functions on $[0,T]$ with
$\omega_{0}=0$. For simplicity, we only consider the relationship between MP
and DPP under the case $d=1$, and the results are similar for $d>1$. The
canonical process $B_{t}(\omega):=\omega_{t}$, for $\omega \in \Omega_{T}$ and
$t\in \lbrack0,T]$. For each given $t\in \lbrack0,T]$, set%
\[
Lip(\Omega_{t}):=\{ \varphi(B_{t_{1}},B_{t_{2}},\ldots,B_{t_{N}}):N\geq
1,t_{1}\leq \cdots \leq t_{N}\leq t,\varphi \in C_{b.Lip}(\mathbb{R}^{N})\},
\]
where $C_{b.Lip}(\mathbb{R}^{N})$ denotes the space of bounded Lipschitz
functions on $\mathbb{R}^{N}$.

For each given $G(a):=\frac{1}{2}(\bar{\sigma}^{2}a^{+}-\underline{\sigma}%
^{2}a^{-})$ for $a\in \mathbb{R}$ with $\bar{\sigma}\geq \underline{\sigma}>0$,
Peng \cite{P07a, P08a} constructed the $G$-expectation sapce $(\Omega
_{T},Lip(\Omega_{T}),(\mathbb{\hat{E}}_{t})_{t\leq T})$, which is a consistent
sublinear expectation space. $\mathbb{\hat{E}}:=\mathbb{\hat{E}}_{0}$ is
called $G$-expectation, $\mathbb{\hat{E}}_{t}$ is called conditional
$G$-expectation and $B$ is called $G$-Brownian motion under $\mathbb{\hat{E}}%
$. Let $\tilde{G}:\mathbb{R}\rightarrow \mathbb{R}$ be a convex function
dominated by $G$ in the following sense:%
\[
\tilde{G}(0)=0,\text{ }\tilde{G}(a_{1})\leq \tilde{G}(a_{2})\text{ if }%
a_{1}\leq a_{2}\text{, }\tilde{G}(a_{3})-\tilde{G}(a_{4})\leq G(a_{3}%
-a_{4})\text{ for }a_{3},a_{4}\in \mathbb{R}\text{.}%
\]
Peng \cite{P2019} constructed a consistent convex $\tilde{G}$-expectation
sapce $(\Omega_{T},Lip(\Omega_{T}),(\mathbb{\tilde{E}}_{t})_{t\leq T})$, which
is dominated by $G$-expectation in the following sense:%
\[
\mathbb{\tilde{E}}_{t}[X]-\mathbb{\tilde{E}}_{t}[Y]\leq \mathbb{\hat{E}}%
_{t}[X-Y]\text{ for each }t\leq T\text{ and }X,Y\in Lip(\Omega_{T}).
\]

For each $t\in \lbrack0,T]$ and $p\geq1$, denote by $L_{G}^{p}(\Omega_{t})$ the
completion of $Lip(\Omega_{t})$ under the norm $||X||_{L_{G}^{p}%
}:=(\mathbb{\hat{E}}[|X|^{p}])^{1/p}$. $\mathbb{\hat{E}}_{t}$ and
$\mathbb{\tilde{E}}_{t}$ can be continuously extended to $L_{G}^{1}(\Omega
_{T})$ under the norm $||\cdot||_{L_{G}^{1}}$. Set%
\[
M_{G}^{0}(0,T)=\left \{  \eta_{t}=\sum_{i=0}^{N-1}\xi_{i}I_{[t_{i},t_{i+1}%
)}(t):\forall N\geq1\text{, }\forall0=t_{0}<\cdots<t_{N}=T\text{, }\forall
\xi_{i}\in Lip(\Omega_{t_{i}})\right \}  .
\]
For each $p\geq1$, denote by $M_{G}^{p}(0,T)$ the completion of $M_{G}%
^{0}(0,T)$ under the norm $||\eta||_{M_{G}^{p}}:=\left(  \mathbb{\hat{E}}%
[\int_{0}^{T}|\eta_{t}|^{p}dt]\right)  ^{1/p}$. For each $\eta \in M_{G}%
^{2}(0,T)$, the $G$-It\^{o} integral $\int_{0}^{T}\eta_{t}dB_{t}$ is well
defined. For any given positive integer $m$, set%
\[
M_{G}^{2}(0,T;\mathbb{R}^{m})=\{(\eta^{1},\ldots,\eta^{m})^{T}:\forall \eta
^{i}\in M_{G}^{2}(0,T),\text{ }i\leq m\}.
\]

The following theorem is the representation theorem of $G$-expectation (see
\cite{DHP11, HP09}).

\begin{theorem}
\label{re-1}There exists a unique convex and weakly compact set of probability
measures $\mathcal{P}$ on $(\Omega_{T},\mathcal{F}_{T})$ such that%
\[
\mathbb{\hat{E}}[X]=\sup_{P\in \mathcal{P}}E_{P}[X]\text{ for all }X\in
L_{G}^{1}(\Omega_{T}),
\]
where $\mathcal{F}_{t}=\sigma(B_{s}:s\leq t)$ for $t\leq T$.
\end{theorem}

For this $\mathcal{P}$, we define capacity%
\[
c(A):=\sup_{P\in \mathcal{P}}P(A)\text{ for }A\in \mathcal{F}_{T}.
\]
A set $A\in \mathcal{F}_{T}$ is polar if $c(A)=0$. A property holds
\textquotedblleft quasi-surely" (q.s. for short) if it holds outside a polar
set. In the following, we do not distinguish two random variables $X$ and $Y$
if $X=Y$ q.s.

The following theorem is the representation theorem of $\tilde{G}$-expectation
(see \cite{LH}).

\begin{theorem}
\label{re-2}Let $\mathcal{P}$ be given in Theorem \ref{re-1}. Then%
\[
\mathbb{\tilde{E}}[X]=\sup_{P\in \mathcal{P}}(E_{P}[X]-\alpha_{0}^{T}(P))\text{
for all }X\in L_{G}^{1}(\Omega_{T}),
\]
where
\[
\alpha_{0}^{T}(P)=\sup_{Y\in L_{G}^{1}(\Omega_{T})}(E_{P}[Y]-\mathbb{\tilde
{E}}[Y]).
\]

\end{theorem}

For simplicity of presentation, we only consider the relationship between MP
and DPP for the following $1$-dimensional forward-backward control system
under $\tilde{G}$-expectation, and the results are similar for other cases.%
\begin{equation}
\left \{
\begin{array}
[c]{l}%
dX_{t}^{u}=h(t,X_{t}^{u},u_{t})d\langle B\rangle_{t}+\sigma(t,X_{t}^{u}%
,u_{t})dB_{t},\text{ }X_{0}^{u}=x_{0}\in \mathbb{R},\\
Y_{t}^{u}=\mathbb{\tilde{E}}_{t}\left[  \Phi(X_{T}^{u})+\int_{t}^{T}%
g(s,X_{s}^{u},Y_{s}^{u},u_{s})d\langle B\rangle_{s}\right]  ,
\end{array}
\right.  \label{nn-e2-1}%
\end{equation}
where the control domain $U\subset \mathbb{R}^{m}$ is a given nonempty convex
and compact set, the set of admissible controls is denoted by
\[
\mathcal{U}[0,T]=M_{G}^{2}(0,T;U)=\{u\in M_{G}^{2}(0,T;\mathbb{R}^{m}%
):u_{t}\in U\text{ for }t\in \lbrack0,T]\},
\]
$\langle B\rangle$ is the quadratic variation of $B$, $h$, $\sigma
:[0,T]\times \mathbb{R}\times U\rightarrow \mathbb{R}$, $\Phi:\mathbb{R}%
\rightarrow \mathbb{R}$ and $g:[0,T]\times \mathbb{R}^{2}\times U\rightarrow
\mathbb{R}$ are continuous functions satisfying the following assumptions:

\begin{description}
\item[(A1)] The derivatives of $h$, $\sigma$, $\Phi$, $g$ in $(x,y,v)$ are
continuous in $(t,x,y,v)$.

\item[(A2)] There exists a constant $L>0$ satisfying%
\[%
\begin{array}
[c]{l}%
|h_{x}(t,x,v)|+|h_{v}(t,x,v)|+|\sigma_{x}(t,x,v)|+|\sigma_{v}(t,x,v)|+|g_{y}%
(t,x,y,v)|\leq L,\\
|g_{x}(t,x,y,v)|+|g_{v}(t,x,y,v)|+|\Phi_{x}(x)|\leq L(1+|x|+|v|).
\end{array}
\]

\end{description}

The cost functional is defined by $J(u):=Y_{0}^{u}$, and the optimal control
problem is to minimize $J(u)$ over $\mathcal{U}[0,T]$. If
\begin{equation}
J(u^{\ast})=\min_{u\in \mathcal{U}[0,T]}J(u), \label{nn-e2-2}%
\end{equation}
then $u^{\ast}$ is called an optimal control. The solution of (\ref{nn-e2-1})
corresponding to $u^{\ast}$ is denoted by $(X^{\ast},Y^{\ast})$. Set%
\begin{equation}
\mathcal{P}^{\ast}=\{P\in \mathcal{P}:E_{P}[\xi^{\ast}]-\alpha_{0}%
^{T}(P)=\mathbb{\tilde{E}}[\xi^{\ast}]\}, \label{nn-e2-3}%
\end{equation}
where $\xi^{\ast}=\Phi(X_{T}^{\ast})+\int_{0}^{T}g(s,X_{s}^{\ast},Y_{s}^{\ast
},u_{s}^{\ast})d\langle B\rangle_{s}$. For each given $P^{\ast}\in
\mathcal{P}^{\ast}$, the adjoint equation%
\begin{equation}
\left \{
\begin{array}
[c]{rl}%
dp_{t}= & -\{[h_{x}(t)+g_{y}(t)]p_{t}+\sigma_{x}(t)q_{t}+g_{x}(t)\}d\langle
B\rangle_{t}+q_{t}dB_{t}+dN_{t},\\
p_{T}= & \Phi_{x}(X_{T}^{\ast}),\text{ }t\in \lbrack0,T],
\end{array}
\right.  \label{nn-e2-5}%
\end{equation}
has a unique square integrable solution $(p,q,N)$ under $P^{\ast}$ (see
\cite{EH}), where $N$ is a martingale orthogonal to $B$ with $N_{0}=0$,
$h_{x}(t)=h_{x}(t,X_{t}^{\ast},u_{t}^{\ast})\,$, similar for $g_{y}(t)$,
$\sigma_{x}(t)$, $g_{x}(t)$ and other cases. Define the Hamiltonian
$H:[0,T]\times \mathbb{R}\times \mathbb{R}\times U\times \mathbb{R}%
\times \mathbb{R}\rightarrow \mathbb{R}$ as follows:%
\begin{equation}
H(t,x,y,v,p,q)=h(t,x,v)p+\sigma(t,x,v)q+g(t,x,y,v). \label{nn-e2-4}%
\end{equation}
The following MP is obtained in \cite{LH}.

\begin{theorem}
\label{nn-MP}Let assumptions (A1)-(A2) hold. If $u^{\ast}$ is an optimal
control for control problem (\ref{nn-e2-2}), then there exists a $P^{\ast}%
\in \mathcal{P}^{\ast}$ such that%
\[
H_{v}(t,X_{t}^{\ast},Y_{t}^{\ast},u_{t}^{\ast},p_{t},q_{t})(v-u_{t}^{\ast
})\geq0,\text{ }\forall v\in U,\text{ a.e., }P^{\ast}\text{-a.s.,}%
\]
where $\mathcal{P}^{\ast}$ is defined in (\ref{nn-e2-3}) and $(p,q,N)$ is the
solution of the adjoint equation (\ref{nn-e2-5}) under $P^{\ast}$.
\end{theorem}

In order to give the DPP of control system (\ref{nn-e2-1}), we define the
value function: for each $(t,x)\in \lbrack0,T]\times \mathbb{R}$,%
\begin{equation}
V(t,x):=\underset{u\in M_{G}^{2}(t,T;U)}{ess\inf}Y_{t}^{t,x,u},
\label{nn-e2-6}%
\end{equation}
where $ess\inf$ is defined in the q.s. sense and $(X_{s}^{t,x,u},Y_{s}%
^{t,x,u})_{s\in \lbrack t,T]}$ satisfies%
\[
\left \{
\begin{array}
[c]{l}%
dX_{s}^{t,x,u}=h(s,X_{s}^{t,x,u},u_{s})d\langle B\rangle_{s}+\sigma
(s,X_{s}^{t,x,u},u_{s})dB_{s},\text{ }X_{t}^{t,x,u}=x,\\
Y_{s}^{t,x,u}=\mathbb{\tilde{E}}_{s}\left[  \Phi(X_{T}^{t,x,u})+\int_{s}%
^{T}g(r,X_{r}^{t,x,u},Y_{r}^{t,x,u},u_{r})d\langle B\rangle_{r}\right]
,\text{ }s\in \lbrack t,T].
\end{array}
\right.
\]

Hu et al. \cite{MJL0} proved that $V(\cdot,\cdot)$ satisfies the DPP and the
following HJB equation.

\begin{theorem}
\label{nn-DPP} Let assumptions (A1)-(A2) hold. Then $V(\cdot,\cdot)$ defined
in (\ref{nn-e2-6}) is the unique viscosity solution of the following HJB
equation:%
\begin{equation}
\partial_{t}V(t,x)+\inf_{v\in U}\tilde{G}(F(t,x,V(t,x),\partial_{x}%
V(t,x),\partial_{xx}^{2}V(t,x),v))=0,\text{ }V(T,x)=\Phi(x), \label{nn-e2-7}%
\end{equation}
where $F:[0,T]\times \mathbb{R}\times \mathbb{R}\times \mathbb{R}\times
\mathbb{R}\times U\rightarrow \mathbb{R}$ is defined by%
\[
F(t,x,a_{1},a_{2},a_{3},v)=\sigma^{2}(t,x,v)a_{3}+2h(t,x,v)a_{2}%
+2g(t,x,a_{1},v).
\]

\end{theorem}

\section{Main results}

For simplicity, the constant $C$ will change from line to line in this
section. The following theorem is our first main result.

\begin{theorem}
\label{MP-DPP-1}Let assumptions (A1)-(A2) hold and let $u^{\ast}$ be an
optimal control for control problem (\ref{nn-e2-2}).

\begin{description}
\item[(i)] If $\partial_{t}V(\cdot,\cdot)$, $\partial_{xx}^{2}V(\cdot,\cdot)$
are continuous and $\partial_{xx}^{2}V(\cdot,\cdot)$ grows polynomially with
respect to $x$, then
\begin{equation}
\mathcal{\bar{P}}^{\ast}:=\{P\in \mathcal{P}:E_{P}[\bar{K}_{T}]-\alpha_{0}%
^{T}(P)=0\} \label{nn-e3-1}%
\end{equation}
is nonempty, and for each given $P^{\ast}\in \mathcal{\bar{P}}^{\ast}$, we have
$P^{\ast}\in \mathcal{P}^{\ast}$ and a.e. $t\in \lbrack0,T]$, $P^{\ast}$-a.s.%
\begin{equation}%
\begin{array}
[c]{l}%
Y_{t}^{\ast}=V(t,X_{t}^{\ast}),\\
-\partial_{t}V(t,X_{t}^{\ast})=\tilde{G}(\bar{F}(t,X_{t}^{\ast},u_{t}^{\ast
}))=\min_{v\in U}\tilde{G}(\bar{F}(t,X_{t}^{\ast},v)),
\end{array}
\label{nn-e3-2}%
\end{equation}

\item where $\mathcal{P}^{\ast}$ is defined in (\ref{nn-e2-3}), $\bar
{F}(t,x,v):=F(t,x,V(t,x),\partial_{x}V(t,x),\partial_{xx}^{2}V(t,x),v)$,%
\begin{equation}
\bar{K}_{t}=\frac{1}{2}\int_{0}^{t}\bar{F}(s,X_{s}^{\ast},u_{s}^{\ast
})d\langle B\rangle_{s}-\int_{0}^{t}\tilde{G}(\bar{F}(s,X_{s}^{\ast}%
,u_{s}^{\ast}))ds\text{, }t\in \lbrack0,T]. \label{nn-e3-3}%
\end{equation}

\item[(ii)] If $\tilde{G}^{\prime}(\cdot)$, $\partial_{tx}^{2}V(\cdot,\cdot)$,
$\partial_{xxx}^{3}V(\cdot,\cdot)$ are continuous and $\partial_{xx}%
^{2}V(\cdot,\cdot)$ grows polynomially with respect to $x$, then for each
given $P^{\ast}\in \mathcal{\bar{P}}^{\ast}$, we have a.e. $t\in \lbrack0,T]$,
$P^{\ast}$-a.s.,%
\begin{equation}
p_{t}=\partial_{x}V(t,X_{t}^{\ast}),\text{ }q_{t}=\partial_{xx}^{2}%
V(t,X_{t}^{\ast})\sigma(t,X_{t}^{\ast},u_{t}^{\ast}),\text{ }N_{t}=0,
\label{nn-e3-4}%
\end{equation}
and the maximum principle%
\begin{equation}
H_{v}(t,X_{t}^{\ast},Y_{t}^{\ast},u_{t}^{\ast},p_{t},q_{t})(v-u_{t}^{\ast
})\geq0,\text{ }\forall v\in U, \label{nn-e3-5}%
\end{equation}
where $(p,q,N)$ satisfies the adjoint equation (\ref{nn-e2-5}) under $P^{\ast
}$, the Hamiltonian $H$ is defined in (\ref{nn-e2-4}).
\end{description}
\end{theorem}

\begin{proof}
(i) It follows from Proposition 3.3 in \cite{LH} that $\mathcal{\bar{P}}%
^{\ast}$ is nonempty. By applying the It\^{o} formula to $V(t,X_{t}^{\ast})$,
we get%
\begin{equation}
d\bar{Y}_{t}=\Sigma_{t}dt-g(t,X_{t}^{\ast},\bar{Y}_{t},u_{t}^{\ast})d\langle
B\rangle_{t}+\bar{Z}_{t}dB_{t}+d\bar{K}_{t},\text{ }\bar{Y}_{T}=\Phi
(X_{T}^{\ast}), \label{nnn-e3-1}%
\end{equation}
where $\bar{K}_{t}$ is defined in (\ref{nn-e3-3}), $\bar{Y}_{t}=V(t,X_{t}%
^{\ast})$, $\bar{Z}_{t}=\partial_{x}V(t,X_{t}^{\ast})\sigma(t,X_{t}^{\ast
},u_{t}^{\ast})$ and
\[
\Sigma_{t}=\partial_{t}V(t,X_{t}^{\ast})+\tilde{G}(\bar{F}(t,X_{t}^{\ast
},u_{t}^{\ast}))\geq0.
\]
Since $\partial_{xx}^{2}V(\cdot,\cdot)$ grows polynomially with respect to
$x$, one can easily check that $\bar{Y}_{t}\in L_{G}^{2}(\Omega_{t})$ and
$\bar{Z}\in M_{G}^{2}(0,T)$. It follows from Lemma 5.7 in \cite{MJL0} and
(\ref{nnn-e3-1}) that%
\begin{equation}
\bar{Y}_{t}=\mathbb{\tilde{E}}_{t}\left[  \Phi(X_{T}^{\ast})-\int_{t}%
^{T}\Sigma_{s}ds+\int_{t}^{T}g(s,X_{s}^{\ast},\bar{Y}_{s},u_{s}^{\ast
})d\langle B\rangle_{s}\right]  ,\text{ }t\in \lbrack0,T]. \label{nnn-e3-2}%
\end{equation}
For each given $P^{\ast}\in \mathcal{\bar{P}}^{\ast}$, by Proposition 3.7 in
\cite{LH}, we have for each $t\in \lbrack0,T]$, $P^{\ast}$-a.s.%
\begin{equation}
\bar{Y}_{t}=E_{P^{\ast}}\left[  \left.  \Phi(X_{T}^{\ast})-\int_{t}^{T}%
\Sigma_{s}ds+\int_{t}^{T}g(s,X_{s}^{\ast},\bar{Y}_{s},u_{s}^{\ast})d\langle
B\rangle_{s}\right \vert \mathcal{F}_{t}\right]  -\alpha_{t}^{T}(P^{\ast}),
\label{nnn-e3-3}%
\end{equation}
where for each $P\in \mathcal{P}$,%
\begin{equation}
\alpha_{t}^{T}(P):=\underset{\xi \in L_{G}^{1}(\Omega_{T})}{ess\sup}^{P}\left(
E_{P}[\xi|\mathcal{F}_{t}]-\mathbb{\tilde{E}}_{t}[\xi]\right)  ,\text{
}P\text{-a.s.} \label{nnn-e3-8}%
\end{equation}
By Theorem 3.4 in \cite{LH}, we have for each $t\in \lbrack0,T]$, $P^{\ast}%
$-a.s.%
\begin{equation}
Y_{t}^{\ast}\geq E_{P^{\ast}}\left[  \left.  \Phi(X_{T}^{\ast})+\int_{t}%
^{T}g(s,X_{s}^{\ast},Y_{s}^{\ast},u_{s}^{\ast})d\langle B\rangle
_{s}\right \vert \mathcal{F}_{t}\right]  -\alpha_{t}^{T}(P^{\ast}).
\label{n-nnn-e3-3}%
\end{equation}
Note that $\exp \left(  \int_{0}^{t}(\int_{0}^{1}g_{y}(s,X_{s}^{\ast},\bar
{Y}_{s}+\alpha(Y_{s}^{\ast}-\bar{Y}_{s}),u_{s}^{\ast})d\alpha)d\langle
B\rangle_{s}\right)  $, $t\in \lbrack0,T]$, has a positive lower bound
$C=\exp(-L\bar{\sigma}^{2}T)$. Then, by using the same method as in the proof
of Theorem 4.4 in \cite{LH} for (\ref{nnn-e3-3}) and (\ref{n-nnn-e3-3}), we
can deduce%
\begin{equation}
Y_{0}^{\ast}-\bar{Y}_{0}\geq C\left(  E_{P^{\ast}}\left[  \int_{0}^{T}%
\Sigma_{t}dt\right]  +\mathbb{\tilde{E}}[\xi^{\ast}]-E_{P^{\ast}}[\xi^{\ast
}]+\alpha_{0}^{T}(P^{\ast})\right)  , \label{nnn-e3-4}%
\end{equation}
where $\xi^{\ast}$ is defined in (\ref{nn-e2-3}). Since $\Sigma_{t}\geq0$,
$\mathbb{\tilde{E}}[\xi^{\ast}]-E_{P^{\ast}}[\xi^{\ast}]+\alpha_{0}%
^{T}(P^{\ast})\geq0$ and $Y_{0}^{\ast}=\bar{Y}_{0}=V(0,x_{0})$, we obtain that
$P^{\ast}\in \mathcal{P}^{\ast}$ and $\Sigma_{t}=0$, a.e. $t\in \lbrack0,T]$,
$P^{\ast}$-a.s., which implies (\ref{nn-e3-2}) by Theorem \ref{nn-DPP} and
(\ref{nnn-e3-3}).

(ii) For each given $P^{\ast}\in \mathcal{\bar{P}}^{\ast}$, by Proposition 5.1
in \cite{LH}, we have%
\begin{equation}
\gamma_{t}=2\tilde{G}^{\prime}(\bar{F}(t,X_{t}^{\ast},u_{t}^{\ast})),\text{
a.e. }t\in \lbrack0,T],\text{ }P^{\ast}\text{-a.s.,} \label{nnn-e3-5}%
\end{equation}
where $d\langle B\rangle_{t}=\gamma_{t}dt$. By Theorem \ref{nn-DPP} and
(\ref{nn-e3-2}), we know that $\partial_{t}V(t,x)+\tilde{G}(\bar{F}%
(t,x,u_{t}^{\ast}))\geq0$ and $\partial_{t}V(t,X_{t}^{\ast})+\tilde{G}(\bar
{F}(t,X_{t}^{\ast},u_{t}^{\ast}))=0$, a.e. $t\in \lbrack0,T]$, $P^{\ast}$-a.s.
Then we obtain a.e. $t\in \lbrack0,T]$, $P^{\ast}$-a.s.,%
\begin{equation}
\partial_{tx}^{2}V(t,X_{t}^{\ast})+\tilde{G}^{\prime}(\bar{F}(t,X_{t}^{\ast
},u_{t}^{\ast}))\partial_{x}\bar{F}(t,X_{t}^{\ast},u_{t}^{\ast})=0.
\label{nnn-e3-7}%
\end{equation}
By applying the It\^{o} formula to $\partial_{x}V(t,X_{t}^{\ast})$, we get%
\begin{equation}
d\tilde{Y}_{t}=\partial_{tx}^{2}V(t,X_{t}^{\ast})dt+[\partial_{xx}%
^{2}V(t,X_{t}^{\ast})h(t,X_{t}^{\ast},u_{t}^{\ast})+\frac{1}{2}\partial
_{xxx}^{3}V(t,X_{t}^{\ast})\sigma^{2}(t,X_{t}^{\ast},u_{t}^{\ast})]d\langle
B\rangle_{t}+\tilde{Z}_{t}dB_{t}, \label{nnn-e3-6}%
\end{equation}
where $\tilde{Y}_{t}=\partial_{x}V(t,X_{t}^{\ast})$ and $\tilde{Z}%
_{t}=\partial_{xx}^{2}V(t,X_{t}^{\ast})\sigma(t,X_{t}^{\ast},u_{t}^{\ast})$.
Note that $\partial_{xx}^{2}V(\cdot,\cdot)$ grows polynomially with respect to
$x$. Then one can verify that (\ref{nn-e3-4}) is the square integrable
solution of the adjoint equation (\ref{nn-e2-5}) under $P^{\ast}$ by
(\ref{nnn-e3-5})-(\ref{nnn-e3-6}). On the other hand, it follows from
(\ref{nn-e3-2}) that $\partial_{v}\tilde{G}(\bar{F}(t,X_{t}^{\ast
},v))|_{v=u_{t}^{\ast}}(v-u_{t}^{\ast})\geq0$, $\forall v\in U$, a.e.
$t\in \lbrack0,T]$, $P^{\ast}$-a.s., which implies (\ref{nn-e3-5}) by
$2\tilde{G}^{\prime}(\cdot)\geq \underline{\sigma}^{2}$.
\end{proof}

\begin{remark}
It is possible that $Y_{t}^{\ast}\not =V(t,X_{t}^{\ast})$ under some
$P\in \mathcal{P}^{\ast}$ (see Example 3 in \cite{L1} under the sublinear case).
\end{remark}

In order to study the relationship between MP and DPP in the viscosity sense,
we need the following propositions. For each given $P\in \mathcal{P}$ and
$t\in \lbrack0,T]$, set%
\begin{equation}
\mathcal{P}(t,P)=\{Q\in \mathcal{P}:Q(A)=P(A)\text{ for }A\in \mathcal{F}%
_{t}\}=\{Q\in \mathcal{P}:E_{Q}[\xi]=E_{P}[\xi]\text{ for }\xi \in
Lip(\Omega_{t})\}, \label{nn-e3-6}%
\end{equation}
we have the following representation theorem for $\mathbb{\tilde{E}}_{t}$ (see
Theorem 3.4 in \cite{LH}):%
\begin{equation}
\mathbb{\tilde{E}}_{t}[\xi]=\underset{Q\in \mathcal{P}(t,P)}{ess\sup}%
^{P}\left(  E_{Q}[\xi|\mathcal{F}_{t}]-\alpha_{t}^{T}(Q)\right)  ,\text{
}P\text{-a.s., for each }\xi \in L_{G}^{1}(\Omega_{T}), \label{nn-e3-9}%
\end{equation}
where $\alpha_{t}^{T}(Q)$ is defined in (\ref{nnn-e3-8}). For each fixed
$\xi \in$ $L_{G}^{1}(\Omega_{T})$, set%
\begin{equation}
\mathcal{P}_{\xi}^{\ast}(t,P)=\{Q\in \mathcal{P}(t,P):E_{Q}[\xi|\mathcal{F}%
_{t}]-\alpha_{t}^{T}(Q)=\mathbb{\tilde{E}}_{t}[\xi],\text{ }P\text{-a.s.}\}.
\label{nn-e3-7}%
\end{equation}

\begin{proposition}
\label{pro-1}Let $\{ \xi_{n}:n\geq1\} \subseteq L_{G}^{1}(\Omega_{T})$ satisfy
$\mathbb{\hat{E}}[|\xi_{n}-\xi|]\rightarrow0$ as $n\rightarrow \infty$ and let
$P\in \mathcal{P}$ be given.

\begin{description}
\item[(i)] For each fixed $t\in \lbrack0,T]$, $\mathcal{P}_{\xi}^{\ast}(t,P)$
is a nonempty convex and weakly compact set.

\item[(ii)] If $\{Q_{n}\in \mathcal{P}_{\xi_{n}}^{\ast}(t,P):n\geq1\}$
converges weakly to $Q$, then $Q\in \mathcal{P}_{\xi}^{\ast}(t,P)$.

\item[(iii)] If $Q\in \mathcal{P}_{\xi}^{\ast}(t,P)$, then $E_{Q}%
[\xi|\mathcal{F}_{s}]-\alpha_{s}^{T}(Q)=\mathbb{\tilde{E}}_{s}[\xi]$,
$Q$-a.s., for $s\in \lbrack t,T]$.
\end{description}
\end{proposition}

\begin{proof}
(i) For each given $Q_{1}$, $Q_{2}\in \mathcal{P}(t,P)$, set $A=\{E_{Q_{1}}%
[\xi|\mathcal{F}_{t}]-\alpha_{t}^{T}(Q_{1})>E_{Q_{2}}[\xi|\mathcal{F}%
_{t}]-\alpha_{t}^{T}(Q_{2})\} \in \mathcal{F}_{t}$. By Lemma 22 in \cite{HP21},
we can find a $Q_{3}\in \mathcal{P}(t,P)$ such that%
\[
E_{Q_{3}}[X|\mathcal{F}_{t}]=E_{Q_{1}}[X|\mathcal{F}_{t}]I_{A}+E_{Q_{2}%
}[X|\mathcal{F}_{t}]I_{A^{c}},\text{ }P\text{-a.s., for }X\in L_{G}^{1}%
(\Omega_{T}).
\]
It is easy to verify that $\alpha_{t}^{T}(Q_{3})=\alpha_{t}^{T}(Q_{1}%
)I_{A}+\alpha_{t}^{T}(Q_{2})I_{A^{c}}$, $P$-a.s. Thus%
\[
E_{Q_{3}}[\xi|\mathcal{F}_{t}]-\alpha_{t}^{T}(Q_{3})=(E_{Q_{1}}[\xi
|\mathcal{F}_{t}]-\alpha_{t}^{T}(Q_{1}))\vee(E_{Q_{2}}[\xi|\mathcal{F}%
_{t}]-\alpha_{t}^{T}(Q_{2})),\text{ }P\text{-a.s.}%
\]
From this we can find a sequence $\{Q_{n}^{\prime}:n\geq1\} \subseteq
\mathcal{P}(t,P)$ such that $\{Q_{n}^{\prime}:n\geq1\}$ converges weakly to
$Q$ and%
\begin{equation}
E_{P}\left[  E_{Q_{n}^{\prime}}[\xi|\mathcal{F}_{t}]-\alpha_{t}^{T}%
(Q_{n}^{\prime})\right]  =E_{Q_{n}^{\prime}}[\xi]-E_{Q_{n}^{\prime}}\left[
\alpha_{t}^{T}(Q_{n}^{\prime})\right]  \uparrow E_{P}\left[  \mathbb{\tilde
{E}}_{t}[\xi]\right]  \text{.} \label{nn-e3-8}%
\end{equation}
It is easy to check that $Q\in \mathcal{P}(t,P)$ by (\ref{nn-e3-6}). By
Proposition 3.5 in \cite{LH}, we can deduce%
\begin{equation}
E_{Q}\left[  \alpha_{t}^{T}(Q)\right]  \leq \underset{n\rightarrow \infty}%
{\lim \inf}E_{Q_{n}^{\prime}}\left[  \alpha_{t}^{T}(Q_{n}^{\prime})\right]  ,
\label{nn-e3-10}%
\end{equation}
which implies $E_{P}\left[  E_{Q}[\xi|\mathcal{F}_{t}]-\alpha_{t}%
^{T}(Q)\right]  =E_{Q}[\xi]-E_{Q}\left[  \alpha_{t}^{T}(Q)\right]  \geq
E_{P}\left[  \mathbb{\tilde{E}}_{t}[\xi]\right]  $ by (\ref{nn-e3-8}). Thus we
obtain $Q\in \mathcal{P}_{\xi}^{\ast}(t,P)$ by (\ref{nn-e3-9}). By
(\ref{nnn-e3-8}) and (\ref{nn-e3-10}), it is easy to check that $\mathcal{P}%
_{\xi}^{\ast}(t,P)$ is convex and weakly compact.

(ii) It is easy to check that $Q\in \mathcal{P}(t,P)$ by (\ref{nn-e3-6}). Since
$E_{Q_{n}}[|\xi_{n}-\xi|]+E_{P}[|\mathbb{\tilde{E}}_{t}[\xi_{n}%
]-\mathbb{\tilde{E}}_{t}[\xi]|]\leq2\mathbb{\hat{E}}[|\xi_{n}-\xi|]$, we can
obtain $E_{P}\left[  E_{Q}[\xi|\mathcal{F}_{t}]-\alpha_{t}^{T}(Q)\right]  \geq
E_{P}\left[  \mathbb{\tilde{E}}_{t}[\xi]\right]  $ by (\ref{nn-e3-10}), which
implies $Q\in \mathcal{P}_{\xi}^{\ast}(t,P)$.

(iii) For each given $s\in \lbrack t,T]$, set
\[
\alpha_{t}^{s}(Q)=\underset{X\in L_{G}^{1}(\Omega_{s})}{ess\sup}^{Q}\left(
E_{Q}[X|\mathcal{F}_{t}]-\mathbb{\tilde{E}}_{t}[X]\right)  ,\text{
}Q\text{-a.s.,}%
\]
it follows from Remark 3.6 in \cite{LH} that $\alpha_{t}^{T}(Q)=\alpha_{t}%
^{s}(Q)+E_{Q}[\alpha_{s}^{T}(Q)|\mathcal{F}_{t}]$, $Q$-a.s. Since
$Q\in \mathcal{P}_{\xi}^{\ast}(t,P)$,%
\[
\mathbb{\tilde{E}}_{s}[\xi]\geq E_{Q}[\xi|\mathcal{F}_{s}]-\alpha_{s}%
^{T}(Q)\text{ and }\mathbb{\tilde{E}}_{t}[\xi]=\mathbb{\tilde{E}}%
_{t}[\mathbb{\tilde{E}}_{s}[\xi]]\geq E_{Q}[\mathbb{\tilde{E}}_{s}%
[\xi]|\mathcal{F}_{t}]-\alpha_{t}^{s}(Q)\text{, }Q\text{-a.s.,}%
\]
we obtain $E_{Q}[\xi|\mathcal{F}_{s}]-\alpha_{s}^{T}(Q)=\mathbb{\tilde{E}}%
_{s}[\xi]$, $Q$-a.s.
\end{proof}

For each fixed $t\in \lbrack0,T]$, set%
\begin{equation}
\mathcal{\tilde{P}}_{t}^{\ast}=\{P\in \mathcal{P}:Y_{t}^{\ast}=V(t,X_{t}^{\ast
}),\text{ }P\text{-a.s.}\}, \label{nn-e3-11}%
\end{equation}
where $(X^{\ast},Y^{\ast})$ is the solution of (\ref{nn-e2-1}) corresponding
to the optimal control $u^{\ast}$ and $V(\cdot,\cdot)$ is defined in
(\ref{nn-e2-6}).

\begin{proposition}
\label{pro-2} $\mathcal{\tilde{P}}_{t}^{\ast}$ defined in (\ref{nn-e3-11}) is
a nonempty convex and weakly compact set.
\end{proposition}

\begin{proof}
Consider the following BSDE under $\tilde{G}$-expectation:%
\[
Y_{s}=\mathbb{\tilde{E}}_{s}\left[  V(t,X_{t}^{\ast})+\int_{s}^{t}%
g(r,X_{r}^{\ast},Y_{r},u_{r}^{\ast})d\langle B\rangle_{r}\right]  ,\text{
}s\in \lbrack0,t].
\]
By (i) in Proposition \ref{pro-1}, there exists a $P\in \mathcal{P}$ such that
$E_{P}[\xi^{\prime}]-\alpha_{0}^{t}(P)=\mathbb{\tilde{E}}[\xi^{\prime}]$,
where $\xi^{\prime}=V(t,X_{t}^{\ast})+\int_{0}^{t}g(r,X_{r}^{\ast},Y_{r}%
,u_{r}^{\ast})d\langle B\rangle_{r}$. It follows from (iii) in Proposition
\ref{pro-1} that%
\begin{equation}
Y_{s}=E_{P}\left[  \left.  V(t,X_{t}^{\ast})+\int_{s}^{t}g(r,X_{r}^{\ast
},Y_{r},u_{r}^{\ast})d\langle B\rangle_{r}\right \vert \mathcal{F}_{s}\right]
-\alpha_{s}^{t}(P),\text{ }P\text{-a.s.} \label{nn-e3-12}%
\end{equation}
By (\ref{nn-e3-9}), we have%
\begin{equation}
Y_{s}^{\ast}\geq E_{P}\left[  \left.  Y_{t}^{\ast}+\int_{s}^{t}g(r,X_{r}%
^{\ast},Y_{r}^{\ast},u_{r}^{\ast})d\langle B\rangle_{r}\right \vert
\mathcal{F}_{s}\right]  -\alpha_{s}^{t}(P),\text{ }P\text{-a.s.}
\label{nn-e3-13}%
\end{equation}
By using the same method as in the proof of Theorem 4.4 in \cite{LH} for
(\ref{nn-e3-12}) and (\ref{nn-e3-13}), we can deduce that%
\begin{equation}
Y_{0}^{\ast}-Y_{0}\geq E_{P}\left[  (Y_{t}^{\ast}-V(t,X_{t}^{\ast}%
))\exp \left(  \int_{0}^{t}\int_{0}^{1}g_{y}(r,X_{r}^{\ast},Y_{r}+\alpha
(Y_{r}^{\ast}-Y_{r}),u_{r}^{\ast})d\alpha d\langle B\rangle_{r}\right)
\right]  . \label{nn-e3-14}%
\end{equation}
It follows from Theorem 4.5 in \cite{MJL0} that $Y_{t}^{\ast}\geq
V(t,X_{t}^{\ast})$, q.s. By Theorem 4.10 (DPP) in \cite{MJL0}, we have
$Y_{0}\geq V(0,x_{0})=Y_{0}^{\ast}$. Thus we obtain $Y_{t}^{\ast}%
=V(t,X_{t}^{\ast}),$ $P$-a.s., from (\ref{nn-e3-14}). Note that $Y_{t}^{\ast
}=V(t,X_{t}^{\ast}),$ $P$-a.s., is equivalent to $E_{P}[Y_{t}^{\ast}%
-V(t,X_{t}^{\ast})]=0$. Then it is easy to verify that $\mathcal{\tilde{P}%
}_{t}^{\ast}$ is convex and weakly compact from this point.
\end{proof}

The following definition of the first-order super-jet and sub-jet of
$V(\cdot,\cdot)$ in $x$ can be found in \cite{CIP}. For each fixed
$(t,x)\in \lbrack0,T]\times \mathbb{R}$,%
\[
\left \{
\begin{array}
[c]{rl}%
D_{x}^{1,+}V(t,x)= & \{a\in \mathbb{R}:V(t,x^{\prime})\leq V(t,x)+a(x^{\prime
}-x)+o(|x^{\prime}-x|)\text{ as }x^{\prime}\rightarrow x\},\\
D_{x}^{1,-}V(t,x)= & \{a\in \mathbb{R}:V(t,x^{\prime})\geq V(t,x)+a(x^{\prime
}-x)+o(|x^{\prime}-x|)\text{ as }x^{\prime}\rightarrow x\}.
\end{array}
\right.
\]
The following theorem is our second main result.

\begin{theorem}
\label{MP-DPP-2}Let assumptions (A1)-(A2) hold and let $u^{\ast}$ be an
optimal control for control problem (\ref{nn-e2-2}). Then, for each fixed
$t\in \lbrack0,T]$ and $P\in \mathcal{\tilde{P}}_{t}^{\ast}$, we have%
\[
D_{x}^{1,-}V(t,X_{t}^{\ast})\subseteq \lbrack p_{t},\bar{p}_{t}],\text{
}P\text{-a.s.,}%
\]
where
\begin{equation}
p_{t}=\underset{Q\in \mathcal{P}^{\ast}(t,P)}{ess\inf}^{P}p_{t}^{Q},\text{
}\bar{p}_{t}=\underset{Q\in \mathcal{P}^{\ast}(t,P)}{ess\sup}^{P}p_{t}%
^{Q},\text{ }P\text{-a.s.,} \label{nn-e3-15}%
\end{equation}
$\mathcal{\tilde{P}}_{t}^{\ast}$ is defined in (\ref{nn-e3-11}),
$\mathcal{P}^{\ast}(t,P)$ is defined in (\ref{nn-e3-7}) corresponding to
$\xi_{t}^{\ast}=\Phi(X_{T}^{\ast})+\int_{t}^{T}g(s,X_{s}^{\ast},Y_{s}^{\ast
},u_{s}^{\ast})d\langle B\rangle_{s}$, $(p^{Q},q^{Q},N^{Q})$ is the solution
of the adjoint equation (\ref{nn-e2-5}) under $Q$. Furthermore, if $p_{t}%
=\bar{p}_{t}$, $P$-a.s., then
\[
\{p_{t}\} \subseteq D_{x}^{1,+}V(t,X_{t}^{\ast}),\text{ }P\text{-a.s.}%
\]

\end{theorem}

\begin{proof}
For each given $\varepsilon \in \lbrack-1,1]$, consider the following
forward-backward SDE under $\tilde{G}$-expectation:%
\[
\left \{
\begin{array}
[c]{l}%
dX_{s}^{\varepsilon}=h(s,X_{s}^{\varepsilon},u_{s}^{\ast})d\langle
B\rangle_{s}+\sigma(s,X_{s}^{\varepsilon},u_{s}^{\ast})dB_{s},\text{ }%
X_{t}^{\varepsilon}=X_{t}^{\ast}+\varepsilon,\\
Y_{s}^{\varepsilon}=\mathbb{\tilde{E}}_{s}\left[  \Phi(X_{T}^{\varepsilon
})+\int_{s}^{T}g(r,X_{r}^{\varepsilon},Y_{r}^{\varepsilon},u_{r}^{\ast
})d\langle B\rangle_{r}\right]  ,\text{ }s\in \lbrack t,T].
\end{array}
\right.
\]
It is clear that $X_{s}^{0}=X_{s}^{\ast}$ and $Y_{s}^{0}=Y_{s}^{\ast}$ for
$s\in \lbrack t,T]$. For each given $\varepsilon^{\prime}$, $\varepsilon
\in \lbrack-1,1]$, set%
\[
\hat{X}_{s}^{\varepsilon^{\prime},\varepsilon}=X_{s}^{\varepsilon^{\prime}%
}-X_{s}^{\varepsilon},\text{ }\hat{Y}_{s}^{\varepsilon^{\prime},\varepsilon
}=Y_{s}^{\varepsilon^{\prime}}-Y_{s}^{\varepsilon},\text{ }s\in \lbrack t,T],
\]
and denote $\hat{X}_{s}^{\varepsilon}=\hat{X}_{s}^{\varepsilon,0}$, $\hat
{Y}_{s}^{\varepsilon}=\hat{Y}_{s}^{\varepsilon,0}$ for $s\in \lbrack t,T]$. The
proof is divided into seven steps.

Step 1: Estimates for $\hat{X}^{\varepsilon^{\prime},\varepsilon}$ and
$\hat{Y}^{\varepsilon^{\prime},\varepsilon}$.

By the estimates of forward-backward SDE under $\tilde{G}$-expectation (see
Theorem 3.2 and the proof of Lemma 4.2 in \cite{MJL0}), we can deduce%
\begin{equation}
\sup_{s\in \lbrack t,T]}\mathbb{\hat{E}}_{t}[|\hat{X}_{s}^{\varepsilon^{\prime
},\varepsilon}|^{4}]\leq C|\varepsilon^{\prime}-\varepsilon|^{4},\text{ }%
\sup_{s\in \lbrack t,T]}\mathbb{\hat{E}}_{t}[|X_{s}^{\varepsilon}|^{4}]\leq
C(1+|X_{t}^{\ast}|^{4}), \label{nn-e3-16}%
\end{equation}%
\begin{equation}
\sup_{s\in \lbrack t,T]}\mathbb{\hat{E}}_{t}[|\hat{Y}_{s}^{\varepsilon^{\prime
},\varepsilon}|^{2}]\leq C(1+|X_{t}^{\ast}|^{2})|\varepsilon^{\prime
}-\varepsilon|^{2},\text{ }\sup_{s\in \lbrack t,T]}\mathbb{\hat{E}}_{t}%
[|Y_{s}^{\varepsilon}|^{2}]\leq C(1+|X_{t}^{\ast}|^{4}), \label{nn-e3-17}%
\end{equation}
where the constant $C>0$ is dependent on $T$, $L$, $\bar{\sigma}^{2}$ and is
independent of $\varepsilon^{\prime}$, $\varepsilon$.

Step 2: Relation between $\hat{Y}^{\varepsilon^{\prime},\varepsilon}$ and
$\hat{X}^{\varepsilon^{\prime},\varepsilon}$.

For each given
\begin{equation}
Q_{\varepsilon}\in \mathcal{P}_{\varepsilon}^{\ast}(t,P):=\{Q\in \mathcal{P}%
(t,P):E_{Q}[\xi_{\varepsilon}|\mathcal{F}_{t}]-\alpha_{t}^{T}%
(Q)=\mathbb{\tilde{E}}_{t}[\xi_{\varepsilon}],\text{ }P\text{-a.s.}\}
\label{new-nnn-e3-1}%
\end{equation}
with $\xi_{\varepsilon}=\Phi(X_{T}^{\varepsilon})+\int_{t}^{T}g(r,X_{r}%
^{\varepsilon},Y_{r}^{\varepsilon},u_{r}^{\ast})d\langle B\rangle_{r}$, it
follows from (iii) in Proposition \ref{pro-1} and (\ref{nn-e3-9}) that for
each $s\in \lbrack t,T]$,
\begin{equation}
Y_{s}^{\varepsilon}=E_{Q_{\varepsilon}}\left[  \left.  \Phi(X_{T}%
^{\varepsilon})+\int_{s}^{T}g(r,X_{r}^{\varepsilon},Y_{r}^{\varepsilon}%
,u_{r}^{\ast})d\langle B\rangle_{r}\right \vert \mathcal{F}_{s}\right]
-\alpha_{s}^{T}(Q_{\varepsilon}),\text{ }Q_{\varepsilon}\text{-a.s.,}
\label{nn-e3-18}%
\end{equation}%
\begin{equation}
Y_{s}^{\varepsilon^{\prime}}\geq E_{Q_{\varepsilon}}\left[  \left.  \Phi
(X_{T}^{\varepsilon^{\prime}})+\int_{s}^{T}g(r,X_{r}^{\varepsilon^{\prime}%
},Y_{r}^{\varepsilon^{\prime}},u_{r}^{\ast})d\langle B\rangle_{r}\right \vert
\mathcal{F}_{s}\right]  -\alpha_{s}^{T}(Q_{\varepsilon}),\text{ }%
Q_{\varepsilon}\text{-a.s.} \label{nn-e3-19}%
\end{equation}
By using the same method as in the proof of Theorem 4.4 in \cite{LH} for
(\ref{nn-e3-18}) and (\ref{nn-e3-19}), we obtain%
\begin{equation}
\hat{Y}_{t}^{\varepsilon^{\prime},\varepsilon}\geq E_{Q_{\varepsilon}}\left[
\left.  \Phi_{x}(X_{T}^{\ast})\hat{X}_{T}^{\varepsilon^{\prime},\varepsilon
}\Lambda_{T}+I_{\Phi}^{\varepsilon^{\prime},\varepsilon}\Lambda_{T}+\int
_{t}^{T}(g_{x}(s)\hat{X}_{s}^{\varepsilon^{\prime},\varepsilon}+I_{g}%
^{\varepsilon^{\prime},\varepsilon}(s))\Lambda_{s}d\langle B\rangle
_{s}\right \vert \mathcal{F}_{t}\right]  ,\text{ }P\text{-a.s.,}
\label{nn-e3-20}%
\end{equation}
where $g_{x}^{\varepsilon}(s)=g_{x}(s,X_{s}^{\varepsilon},Y_{s}^{\varepsilon
},u_{s}^{\ast})$, similar for $g_{y}^{\varepsilon}(s)$, $h_{x}^{\varepsilon
}(s)$ and $\sigma_{x}^{\varepsilon}(s)$, $g_{x}(s)$, $g_{y}(s)$, $h_{x}(s)$
and $\sigma_{x}(s)$ are defined in (\ref{nn-e2-5}),%
\[%
\begin{array}
[c]{rl}%
\Lambda_{s}= & \exp \left(  \int_{t}^{s}g_{y}(r)d\langle B\rangle_{r}\right)
,\\
I_{\Phi}^{\varepsilon^{\prime},\varepsilon}= & \hat{X}_{T}^{\varepsilon
^{\prime},\varepsilon}\int_{0}^{1}(\Phi_{x}(X_{T}^{\varepsilon}+\alpha \hat
{X}_{T}^{\varepsilon^{\prime},\varepsilon})-\Phi_{x}(X_{T}^{\varepsilon
}))d\alpha+\hat{X}_{T}^{\varepsilon^{\prime},\varepsilon}(\Phi_{x}%
(X_{T}^{\varepsilon})-\Phi_{x}(X_{T}^{\ast})),\\
I_{g}^{\varepsilon^{\prime},\varepsilon}(s)= & \hat{X}_{s}^{\varepsilon
^{\prime},\varepsilon}\int_{0}^{1}(g_{x}(s,X_{s}^{\varepsilon}+\alpha \hat
{X}_{s}^{\varepsilon^{\prime},\varepsilon},Y_{s}^{\varepsilon}+\alpha \hat
{Y}_{s}^{\varepsilon^{\prime},\varepsilon},u_{s}^{\ast})-g_{x}^{\varepsilon
}(s))d\alpha+\hat{X}_{s}^{\varepsilon^{\prime},\varepsilon}(g_{x}%
^{\varepsilon}(s)-g_{x}(s))\\
& +\hat{Y}_{s}^{\varepsilon^{\prime},\varepsilon}\int_{0}^{1}(g_{y}%
(s,X_{s}^{\varepsilon}+\alpha \hat{X}_{s}^{\varepsilon^{\prime},\varepsilon
},Y_{s}^{\varepsilon}+\alpha \hat{Y}_{s}^{\varepsilon^{\prime},\varepsilon
},u_{s}^{\ast})-g_{y}^{\varepsilon}(s))d\alpha+\hat{Y}_{s}^{\varepsilon
^{\prime},\varepsilon}(g_{y}^{\varepsilon}(s)-g_{y}(s)).
\end{array}
\]
Applying the It\^{o} formula to $\Lambda_{s}p_{s}^{Q_{\varepsilon}}\hat{X}%
_{s}^{\varepsilon^{\prime},\varepsilon}$ under $Q_{\varepsilon}$, we obtain
$P$-a.s.%
\begin{equation}
p_{t}^{Q_{\varepsilon}}\hat{X}_{t}^{\varepsilon^{\prime},\varepsilon
}=E_{Q_{\varepsilon}}\left[  \left.  \Phi_{x}(X_{T}^{\ast})\hat{X}%
_{T}^{\varepsilon^{\prime},\varepsilon}\Lambda_{T}+\int_{t}^{T}(g_{x}%
(s)\hat{X}_{s}^{\varepsilon^{\prime},\varepsilon}-p_{s}^{Q_{\varepsilon}}%
I_{h}^{\varepsilon^{\prime},\varepsilon}(s)-q_{s}^{Q_{\varepsilon}}I_{\sigma
}^{\varepsilon^{\prime},\varepsilon}(s))\Lambda_{s}d\langle B\rangle
_{s}\right \vert \mathcal{F}_{t}\right]  , \label{nn-e3-22}%
\end{equation}
where%
\[%
\begin{array}
[c]{rl}%
I_{h}^{\varepsilon^{\prime},\varepsilon}(s)= & \hat{X}_{s}^{\varepsilon
^{\prime},\varepsilon}\int_{0}^{1}(h_{x}(s,X_{s}^{\varepsilon}+\alpha \hat
{X}_{s}^{\varepsilon^{\prime},\varepsilon},u_{s}^{\ast})-h_{x}^{\varepsilon
}(s))d\alpha+\hat{X}_{s}^{\varepsilon^{\prime},\varepsilon}(h_{x}%
^{\varepsilon}(s)-h_{x}(s)),\\
I_{\sigma}^{\varepsilon^{\prime},\varepsilon}(s)= & \hat{X}_{s}^{\varepsilon
^{\prime},\varepsilon}\int_{0}^{1}(\sigma_{x}(s,X_{s}^{\varepsilon}+\alpha
\hat{X}_{s}^{\varepsilon^{\prime},\varepsilon},u_{s}^{\ast})-\sigma
_{x}^{\varepsilon}(s))d\alpha+\hat{X}_{s}^{\varepsilon^{\prime},\varepsilon
}(\sigma_{x}^{\varepsilon}(s)-\sigma_{x}(s)).
\end{array}
\]
It follows from (\ref{nn-e3-20}) and (\ref{nn-e3-22}) that%
\begin{equation}
\hat{Y}_{t}^{\varepsilon^{\prime},\varepsilon}-p_{t}^{Q_{\varepsilon}}\hat
{X}_{t}^{\varepsilon^{\prime},\varepsilon}\geq E_{Q_{\varepsilon}}\left[
\left.  I_{\Phi}^{\varepsilon^{\prime},\varepsilon}\Lambda_{T}+\int_{t}%
^{T}(I_{g}^{\varepsilon^{\prime},\varepsilon}(s)+p_{s}^{Q_{\varepsilon}}%
I_{h}^{\varepsilon^{\prime},\varepsilon}(s)+q_{s}^{Q_{\varepsilon}}I_{\sigma
}^{\varepsilon^{\prime},\varepsilon}(s))\Lambda_{s}d\langle B\rangle
_{s}\right \vert \mathcal{F}_{t}\right]  ,\text{ }P\text{-a.s.}
\label{nn-e3-24}%
\end{equation}

Step 3: Estimates for the terms in the right side of (\ref{nn-e3-24}).

For each given positive integer $N$ and rational number $\delta \in(0,1]$, set
\[%
\begin{array}
[c]{rl}%
\phi_{g_{x}}^{N}(\delta):= & \sup \left \{  |g_{x}(s,x^{\prime},y^{\prime
},v)-g_{x}(s,x,y,v)|:s\in \lbrack0,T],\text{ }v\in U,\right. \\
& \left.  |x|\leq N,\text{ }|y|\leq N,\text{ }|x^{\prime}-x|\leq \delta,\text{
}|y^{\prime}-y|\leq \delta \right \}  ,
\end{array}
\]
similar for $\phi_{g_{y}}^{N}(\delta)$, $\phi_{h_{x}}^{N}(\delta)$,
$\phi_{\sigma_{x}}^{N}(\delta)$ and $\phi_{\Phi_{x}}^{N}(\delta)$. Set
$\phi^{N}(\delta)=\phi_{g_{x}}^{N}(\delta)+\phi_{g_{y}}^{N}(\delta
)+\phi_{h_{x}}^{N}(\delta)+\phi_{\sigma_{x}}^{N}(\delta)+\phi_{\Phi_{x}}%
^{N}(\delta)$, then $\phi^{N}(\delta)\rightarrow0$ as $\delta \rightarrow0$ by
(A1). Note that%
\begin{align*}
&  |g_{x}(s,X_{s}^{\varepsilon}+\alpha \hat{X}_{s}^{\varepsilon^{\prime
},\varepsilon},Y_{s}^{\varepsilon}+\alpha \hat{Y}_{s}^{\varepsilon^{\prime
},\varepsilon},u_{s}^{\ast})-g_{x}^{\varepsilon}(s)|\\
&  \leq \phi^{N}(\delta)+C(1+|X_{s}^{\varepsilon}|+|X_{s}^{\varepsilon^{\prime
}}|)\left(  \sqrt{N^{-1}(|X_{s}^{\varepsilon}|+|Y_{s}^{\varepsilon}|)}%
+\sqrt{\delta^{-1}(|\hat{X}_{s}^{\varepsilon^{\prime},\varepsilon}|+|\hat
{Y}_{s}^{\varepsilon^{\prime},\varepsilon}|)}\right)  ,
\end{align*}
where the constant $C>0$ depends on $L$. Then, by (\ref{nn-e3-16}),
(\ref{nn-e3-17}) and H\"{o}lder's inequality, we obtain%
\begin{equation}%
\begin{array}
[c]{l}%
\left \vert E_{Q_{\varepsilon}}\left[  \left.  \int_{t}^{T}\Lambda_{s}\hat
{X}_{s}^{\varepsilon^{\prime},\varepsilon}\int_{0}^{1}(g_{x}(s,X_{s}%
^{\varepsilon}+\alpha \hat{X}_{s}^{\varepsilon^{\prime},\varepsilon}%
,Y_{s}^{\varepsilon}+\alpha \hat{Y}_{s}^{\varepsilon^{\prime},\varepsilon
},u_{s}^{\ast})-g_{x}^{\varepsilon}(s))d\alpha d\langle B\rangle
_{s}\right \vert \mathcal{F}_{t}\right]  \right \vert \\
\leq C\int_{t}^{T}(\mathbb{\hat{E}}_{t}\left[  |\hat{X}_{s}^{\varepsilon
^{\prime},\varepsilon}|^{2}\right]  )^{1/2}[\phi^{N}(\delta)+(\mathbb{\hat{E}%
}_{t}\left[  |L_{1}^{\varepsilon^{\prime},\varepsilon}(s)|^{4}\right]
)^{1/4}(\mathbb{\hat{E}}_{t}\left[  |L_{2}^{\varepsilon^{\prime},\varepsilon
}(s)|^{4}\right]  )^{1/4}]ds\\
\leq C(1+|X_{t}^{\ast}|^{2})(\phi^{N}(\delta)+\sqrt{1/N}+\sqrt{(|\varepsilon
^{\prime}|+|\varepsilon|)/\delta})|\varepsilon^{\prime}-\varepsilon|,\text{
}P\text{-a.s.,}%
\end{array}
\label{nn-e3-25}%
\end{equation}
where $L_{1}^{\varepsilon^{\prime},\varepsilon}(s)=1+|X_{s}^{\varepsilon
}|+|X_{s}^{\varepsilon^{\prime}}|$, $L_{2}^{\varepsilon^{\prime},\varepsilon
}(s)=\sqrt{N^{-1}(|X_{s}^{\varepsilon}|+|Y_{s}^{\varepsilon}|)}+\sqrt
{\delta^{-1}(|\hat{X}_{s}^{\varepsilon^{\prime},\varepsilon}|+|\hat{Y}%
_{s}^{\varepsilon^{\prime},\varepsilon}|)}$, the constant $C>0$ depends on
$T$, $L$ and $\bar{\sigma}^{2}$. By the estimate of BSDE (\ref{nn-e2-5}) (see
\cite{EH}), we have%
\begin{align*}
E_{Q_{\varepsilon}}\left[  \left.  \int_{t}^{T}(|p_{s}^{Q_{\varepsilon}}%
|^{2}+|q_{s}^{Q_{\varepsilon}}|^{2})ds\right \vert \mathcal{F}_{t}\right]   &
\leq CE_{Q_{\varepsilon}}\left[  \left.  |\Phi_{x}(X_{T}^{\ast})|^{2}+\int
_{t}^{T}|g_{x}(s)|^{2}ds\right \vert \mathcal{F}_{t}\right] \\
&  \leq C\left(  1+\mathbb{\hat{E}}_{t}\left[  |X_{T}^{\ast}|^{2}\right]
+\int_{t}^{T}\mathbb{\hat{E}}_{t}\left[  |X_{s}^{\ast}|^{2}\right]  ds\right)
\\
&  \leq C(1+|X_{t}^{\ast}|^{2}),\text{ }P\text{-a.s.,}%
\end{align*}
where the constant $C>0$ depends on $T$, $L$, $\bar{\sigma}^{2}$ and
$\underline{\sigma}^{2}$. Then, by using a similar method in (\ref{nn-e3-25}),
we can deduce%
\begin{equation}%
\begin{array}
[c]{l}%
\left \vert E_{Q_{\varepsilon}}\left[  \left.  I_{\Phi}^{\varepsilon^{\prime
},\varepsilon}\Lambda_{T}+\int_{t}^{T}(I_{g}^{\varepsilon^{\prime}%
,\varepsilon}(s)+p_{s}^{Q_{\varepsilon}}I_{h}^{\varepsilon^{\prime
},\varepsilon}(s)+q_{s}^{Q_{\varepsilon}}I_{\sigma}^{\varepsilon^{\prime
},\varepsilon}(s))\Lambda_{s}d\langle B\rangle_{s}\right \vert \mathcal{F}%
_{t}\right]  \right \vert \\
\leq C(1+|X_{t}^{\ast}|^{2})(\phi^{N}(\delta)+\sqrt{1/N}+\sqrt{(|\varepsilon
^{\prime}|+|\varepsilon|)/\delta})|\varepsilon^{\prime}-\varepsilon|,\text{
}P\text{-a.s.,}%
\end{array}
\label{nn-e3-26}%
\end{equation}
where the constant $C>0$ depends on $T$, $L$, $\bar{\sigma}^{2}$ and
$\underline{\sigma}^{2}$.

Step 4: Relation between $p_{t}^{Q_{\varepsilon^{\prime}}}$ and $p_{t}%
^{Q_{\varepsilon}}$.

Let $\varepsilon^{\prime}>\varepsilon$, and let positive integer $N$ and
rational number $\delta \in(0,1]$ be given. For each given $Q_{\varepsilon}%
\in \mathcal{P}_{\varepsilon}^{\ast}(t,P)$, by (\ref{nn-e3-24}) and
(\ref{nn-e3-26}), we obtain $P$-a.s.
\begin{equation}
\hat{Y}_{t}^{\varepsilon^{\prime},\varepsilon}\geq p_{t}^{Q_{\varepsilon}}%
\hat{X}_{t}^{\varepsilon^{\prime},\varepsilon}-C(1+|X_{t}^{\ast}|^{2}%
)(\phi^{N}(\delta)+\sqrt{1/N}+\sqrt{(|\varepsilon^{\prime}|+|\varepsilon
|)/\delta})|\varepsilon^{\prime}-\varepsilon|. \label{nnn-e3-27}%
\end{equation}
Note that $\hat{X}_{t}^{\varepsilon,\varepsilon^{\prime}}=-\hat{X}%
_{t}^{\varepsilon^{\prime},\varepsilon}$ and $\hat{Y}_{t}^{\varepsilon
,\varepsilon^{\prime}}=-\hat{Y}_{t}^{\varepsilon^{\prime},\varepsilon}$. Then,
for each given $Q_{\varepsilon^{\prime}}\in \mathcal{P}_{\varepsilon^{\prime}%
}^{\ast}(t,P)$, we get $P$-a.s.%
\begin{equation}
\hat{Y}_{t}^{\varepsilon^{\prime},\varepsilon}\leq p_{t}^{Q_{\varepsilon
^{\prime}}}\hat{X}_{t}^{\varepsilon^{\prime},\varepsilon}+C(1+|X_{t}^{\ast
}|^{2})(\phi^{N}(\delta)+\sqrt{1/N}+\sqrt{(|\varepsilon^{\prime}%
|+|\varepsilon|)/\delta})|\varepsilon^{\prime}-\varepsilon| \label{nnn-e3-28}%
\end{equation}
by (\ref{nn-e3-24}) and (\ref{nn-e3-26}). It follows from $\hat{X}%
_{t}^{\varepsilon^{\prime},\varepsilon}=\varepsilon^{\prime}-\varepsilon$,
(\ref{nnn-e3-27}) and (\ref{nnn-e3-28}) that%
\begin{equation}
p_{t}^{Q_{\varepsilon^{\prime}}}\geq p_{t}^{Q_{\varepsilon}}-2C(1+|X_{t}%
^{\ast}|^{2})(\phi^{N}(\delta)+\sqrt{1/N}+\sqrt{(|\varepsilon^{\prime
}|+|\varepsilon|)/\delta})\text{, }P\text{-a.s.} \label{nnn-e3-29}%
\end{equation}

Step 5: The limit of $p_{t}^{Q_{\varepsilon}}$ as $\varepsilon \rightarrow0$.

For each given $\{Q_{n}\in \mathcal{P}_{\varepsilon_{n}}^{\ast}(t,P):n\geq1\}$
with $\varepsilon_{n}\downarrow0$. Since $\mathcal{P}$ is weakly compact, we
can choose a subsequence $\{ \tilde{Q}_{n}\in \mathcal{P}_{\tilde{\varepsilon
}_{n}}^{\ast}(t,P):n\geq1\}$ which converges weakly to $Q^{\ast}$. It is easy
to check that $\mathbb{\hat{E}}[|\xi_{\tilde{\varepsilon}_{n}}-\xi_{t}^{\ast
}|]\rightarrow0$ as $n\rightarrow \infty$. By (ii) of Proposition \ref{pro-1},
we have $Q^{\ast}\in \mathcal{P}^{\ast}(t,P)$. Let positive integer $N$ and
rational number $\delta \in(0,1]$ be given. For each $Q\in \mathcal{P}^{\ast
}(t,P)$, by (\ref{nnn-e3-29}), we obtain%
\[
p_{t}^{Q}\leq p_{t}^{\tilde{Q}_{n}}+2C(1+|X_{t}^{\ast}|^{2})(\phi^{N}%
(\delta)+\sqrt{1/N}+\sqrt{\tilde{\varepsilon}_{n}/\delta}),\text{
}P\text{-a.s.}%
\]
Then we have%
\[
\bar{p}_{t}\leq \underset{n\rightarrow \infty}{\lim \inf}p_{t}^{\tilde{Q}_{n}%
}+2C(1+|X_{t}^{\ast}|^{2})(\phi^{N}(\delta)+\sqrt{1/N}),\text{ }P\text{-a.s.,}%
\]
which implies%
\begin{equation}
\bar{p}_{t}\leq \underset{n\rightarrow \infty}{\lim \inf}p_{t}^{\tilde{Q}_{n}%
},\text{ }P\text{-a.s.,} \label{nnn-e3-30}%
\end{equation}
by taking $\delta \downarrow0$ and then $N\rightarrow \infty$. On the other
hand, for each $Q^{\prime}\in \mathcal{P}(t,P)$, it is easy to verify that%
\[
p_{t}^{Q^{\prime}}=E_{Q^{\prime}}\left[  \Phi_{x}(X_{T}^{\ast})l_{T}+\int
_{t}^{T}g_{x}(s)l_{s}d\langle B\rangle_{s}|\mathcal{F}_{t}\right]  ,\text{
}P\text{-a.s.,}%
\]
where $l_{s}=\exp \left(  \int_{t}^{s}\sigma_{x}(r)dB_{r}+\int_{t}^{s}%
(h_{x}(r)+g_{y}(r)-\frac{1}{2}|\sigma_{x}(r)|^{2})d\langle B\rangle
_{r}\right)  $ for $s\in \lbrack t,T]$. Since
\[
\Phi_{x}(X_{T}^{\ast})l_{T}+\int_{t}^{T}g_{x}(s)l_{s}d\langle B\rangle_{s}\in
L_{G}^{1}(\Omega_{T}),
\]
we obtain that, for each positive $\varsigma \in Lip(\Omega_{t})$,
\[
E_{\tilde{Q}_{n}}\left[  \varsigma \left(  \Phi_{x}(X_{T}^{\ast})l_{T}+\int
_{t}^{T}g_{x}(s)l_{s}d\langle B\rangle_{s}\right)  \right]  \rightarrow
E_{Q^{\ast}}\left[  \varsigma \left(  \Phi_{x}(X_{T}^{\ast})l_{T}+\int_{t}%
^{T}g_{x}(s)l_{s}d\langle B\rangle_{s}\right)  \right]
\]
as $n\rightarrow \infty$, which implies $E_{P}\left[  \varsigma p_{t}%
^{\tilde{Q}_{n}}\right]  \rightarrow E_{P}\left[  \varsigma p_{t}^{Q^{\ast}%
}\right]  $ as $n\rightarrow \infty$. Noting that
\[
|p_{t}^{\tilde{Q}_{n}}|\leq \mathbb{\hat{E}}_{t}\left[  \left \vert \Phi
_{x}(X_{T}^{\ast})l_{T}+\int_{t}^{T}g_{x}(s)l_{s}d\langle B\rangle
_{s}\right \vert \right]  ,
\]
by Fatou's lemma, we get%
\[
E_{P}\left[  \varsigma \underset{n\rightarrow \infty}{\lim \inf}p_{t}^{\tilde
{Q}_{n}}\right]  \leq \lim_{n\rightarrow \infty}E_{P}\left[  \varsigma
p_{t}^{\tilde{Q}_{n}}\right]  =E_{P}\left[  \varsigma p_{t}^{Q^{\ast}}\right]
\leq E_{P}\left[  \varsigma \underset{n\rightarrow \infty}{\lim \sup}%
p_{t}^{\tilde{Q}_{n}}\right]
\]
for each positive $\varsigma \in Lip(\Omega_{t})$, which implies%
\begin{equation}
\underset{n\rightarrow \infty}{\lim \inf}p_{t}^{\tilde{Q}_{n}}\leq
p_{t}^{Q^{\ast}}\leq \underset{n\rightarrow \infty}{\lim \sup}p_{t}^{\tilde
{Q}_{n}},\text{ }P\text{-a.s.} \label{nn-e3-30}%
\end{equation}
For each given $m\geq1$, set%
\[
\rho_{t}^{m}=\inf_{k\geq m}p_{t}^{Q_{k}}\text{, }\rho_{t}^{m,n}=\min_{m\leq
k\leq n}p_{t}^{Q_{k}}\text{ for }n\geq m\text{.}%
\]
By (\ref{nnn-e3-29}), we have%
\[
\rho_{t}^{m,n}\geq p_{t}^{Q_{n}}-2C(1+|X_{t}^{\ast}|^{2})(\phi^{N}%
(\delta)+\sqrt{1/N}+\sqrt{\varepsilon_{m}/\delta}),\text{ }P\text{-a.s.,}%
\]
which implies%
\[
\rho_{t}^{m}\geq \underset{n\rightarrow \infty}{\lim \sup}p_{t}^{Q_{n}%
}-2C(1+|X_{t}^{\ast}|^{2})(\phi^{N}(\delta)+\sqrt{1/N}+\sqrt{\varepsilon
_{m}/\delta}),\text{ }P\text{-a.s.}%
\]
Thus%
\[
\underset{n\rightarrow \infty}{\lim \inf}p_{t}^{Q_{n}}\geq \underset
{n\rightarrow \infty}{\lim \sup}p_{t}^{Q_{n}}-2C(1+|X_{t}^{\ast}|^{2})(\phi
^{N}(\delta)+\sqrt{1/N}),\text{ }P\text{-a.s.,}%
\]
which implies%
\begin{equation}
\underset{n\rightarrow \infty}{\lim \inf}p_{t}^{Q_{n}}=\underset{n\rightarrow
\infty}{\lim \sup}p_{t}^{Q_{n}},\text{ }P\text{-a.s.} \label{nnn-e3-31}%
\end{equation}
It follows from $Q^{\ast}\in \mathcal{P}^{\ast}(t,P)$ and (\ref{nnn-e3-30}%
)-(\ref{nnn-e3-31}) that%
\begin{equation}
\bar{p}_{t}=\lim_{n\rightarrow \infty}p_{t}^{Q_{n}},\text{ }P\text{-a.s.}
\label{nnn-e3-32}%
\end{equation}
Similarly, for each given $\{Q_{n}^{\prime}\in \mathcal{P}_{\varepsilon
_{n}^{\prime}}^{\ast}(t,P):n\geq1\}$ with $\varepsilon_{n}^{\prime}\uparrow0$,
we can prove that%
\begin{equation}
p_{t}=\lim_{n\rightarrow \infty}p_{t}^{Q_{n}^{\prime}},\text{ }P\text{-a.s.}
\label{nnn-e3-33}%
\end{equation}

Step 6: $D_{x}^{1,-}V(t,X_{t}^{\ast})\subseteq \lbrack p_{t},\bar{p}_{t}]$, $P$-a.s.

Let $Q_{n}$ and $\varepsilon_{n}$, $n\geq1$, be given in Step 5. It follows
from Theorem 4.5 in \cite{MJL0} that $Y_{t}^{\varepsilon_{n}}\geq
V(t,X_{t}^{\ast}+\varepsilon_{n})$, q.s. Obviously we can find a $\Omega^{1}$
with $P(\Omega^{1})=1$ such that%
\[
Y_{t}^{\ast}(\omega)=V(t,X_{t}^{\ast}(\omega))\text{, }Y_{t}^{\varepsilon_{n}%
}(\omega)\geq V(t,X_{t}^{\ast}(\omega)+\varepsilon_{n})\text{, }\bar{p}%
_{t}(\omega)=\lim_{n\rightarrow \infty}p_{t}^{Q_{n}}(\omega)\text{,}%
\]%
\[
\hat{Y}_{t}^{\varepsilon_{n}}(\omega)\leq p_{t}^{Q_{n}}(\omega)\varepsilon
_{n}+C(1+|X_{t}^{\ast}(\omega)|^{2})(\phi^{N}(\delta)+\sqrt{1/N}%
+\sqrt{\varepsilon_{n}/\delta})\varepsilon_{n}%
\]
for any $\omega \in \Omega^{1}$, $n\geq1$, $N\geq1$ and rational number
$\delta \in(0,1]$. For each given $\omega \in \Omega^{1}$, we have%
\[
a\varepsilon_{n}+o(\varepsilon_{n})\leq V(t,X_{t}^{\ast}(\omega)+\varepsilon
_{n})-V(t,X_{t}^{\ast}(\omega))\leq Y_{t}^{\varepsilon_{n}}(\omega
)-Y_{t}^{\ast}(\omega)
\]
if $a\in D_{x}^{1,-}V(t,X_{t}^{\ast}(\omega))$. Then we get%
\[
a\leq \bar{p}_{t}(\omega)+C(1+|X_{t}^{\ast}(\omega)|^{2})(\phi^{N}%
(\delta)+\sqrt{1/N}),
\]
which implies $a\leq \bar{p}_{t}(\omega)$ by taking $\delta \downarrow0$ and
then $N\rightarrow \infty$. Thus we obtain $D_{x}^{1,-}V(t,X_{t}^{\ast
})\subseteq(-\infty,\bar{p}_{t}]$, $P$-a.s. Let $Q_{n}^{\prime}$ and
$\varepsilon_{n}^{\prime}$, $n\geq1$, be given in Step 5. We can deduce
$D_{x}^{1,-}V(t,X_{t}^{\ast})\subseteq \lbrack p_{t},\infty)$, $P$-a.s., by
using the same method. Thus we get $D_{x}^{1,-}V(t,X_{t}^{\ast})\subseteq
\lbrack p_{t},\bar{p}_{t}]$, $P$-a.s.

Step 7: $\{p_{t}\} \subseteq D_{x}^{1,+}V(t,X_{t}^{\ast})$, $P$-a.s.

Let $A_{+}$ be the set of all rational numbers in $(0,1]$ and let
$Q_{\varepsilon}\in \mathcal{P}_{\varepsilon}^{\ast}(t,P)$ for $\varepsilon \in
A_{+}$. By using the same method as in the proof of (\ref{nnn-e3-32}), we can
get%
\[
\bar{p}_{t}=\lim_{\varepsilon \in A_{+},\varepsilon \downarrow0}p_{t}%
^{Q_{\varepsilon}},\text{ }P\text{-a.s.}%
\]
Since $V(t,X_{t}^{\ast}+\varepsilon)-V(t,X_{t}^{\ast})\leq \hat{Y}%
_{t}^{\varepsilon}$, $P$-a.s., and%
\[
\hat{Y}_{t}^{\varepsilon}\leq p_{t}^{Q_{\varepsilon}}\varepsilon
+C(1+|X_{t}^{\ast}|^{2})(\phi^{N}(\delta)+\sqrt{1/N}+\sqrt{\varepsilon/\delta
})\varepsilon,\text{ }P\text{-a.s.,}%
\]
for each $\varepsilon \in A_{+}$, we obtain%
\[
\underset{\varepsilon \in A_{+},\varepsilon \downarrow0}{\lim \sup}%
\frac{V(t,X_{t}^{\ast}+\varepsilon)-V(t,X_{t}^{\ast})}{\varepsilon}\leq \bar
{p}_{t}+C(1+|X_{t}^{\ast}|^{2})(\phi^{N}(\delta)+\sqrt{1/N}),\text{
}P\text{-a.s.,}%
\]
which implies%
\[
\underset{\varepsilon \in A_{+},\varepsilon \downarrow0}{\lim \sup}%
\frac{V(t,X_{t}^{\ast}+\varepsilon)-V(t,X_{t}^{\ast})}{\varepsilon}\leq \bar
{p}_{t},\text{ }P\text{-a.s.,}%
\]
by taking $\delta \downarrow0$ and then $N\rightarrow \infty$. Noting that
$V(t,\cdot)$ is continuous, we deduce%
\begin{equation}
\underset{\varepsilon \downarrow0}{\lim \sup}\frac{V(t,X_{t}^{\ast}%
+\varepsilon)-V(t,X_{t}^{\ast})}{\varepsilon}\leq \bar{p}_{t},\text{
}P\text{-a.s.} \label{nnn-e3-34}%
\end{equation}
Thus we obtain%
\begin{equation}
V(t,X_{t}^{\ast}+\varepsilon)-V(t,X_{t}^{\ast})\leq \bar{p}_{t}\varepsilon
+o(\varepsilon)\text{ for }\varepsilon>0,\text{ }P\text{-a.s.,}
\label{nnn-e3-35}%
\end{equation}
by (\ref{nnn-e3-34}). Similarly, we can get%
\begin{equation}
V(t,X_{t}^{\ast}+\varepsilon)-V(t,X_{t}^{\ast})\leq p_{t}\varepsilon
+o(\varepsilon)\text{ for }\varepsilon<0,\text{ }P\text{-a.s.}
\label{nnn-e3-36}%
\end{equation}
If $p_{t}=\bar{p}_{t}$, $P$-a.s., then we obtain $\{p_{t}\} \subseteq
D_{x}^{1,+}V(t,X_{t}^{\ast})$, $P$-a.s., by (\ref{nnn-e3-35}) and
(\ref{nnn-e3-36}).
\end{proof}

\begin{remark}
$D_{x}^{1,+}V(t,X_{t}^{\ast})$ may be empty if $p_{t}\not =\bar{p}_{t}$ (see
Example 1 in \cite{L1} under the sublinear case).
\end{remark}

Through the proof of Theorem \ref{MP-DPP-2}, we know that Steps 1-5 hold for
each fixed $P\in \mathcal{P}$, but Steps 6-7 hold only on the set
$\{Y_{t}^{\ast}=V(t,X_{t}^{\ast})\}$, $P$-a.s. So we can obtain the following conclusion.

\begin{corollary}
Let assumptions (A1)-(A2) hold and let $u^{\ast}$ be an optimal control for
control problem (\ref{nn-e2-2}). Then, for each fixed $t\in \lbrack0,T]$ and
$P\in \mathcal{P}$, we have%
\[
D_{x}^{1,-}V(t,X_{t}^{\ast})\subseteq \lbrack p_{t},\bar{p}_{t}]\text{ on
}\{Y_{t}^{\ast}=V(t,X_{t}^{\ast})\},\text{ }P\text{-a.s.,}%
\]%
\[
\{p_{t}\} \subseteq D_{x}^{1,+}V(t,X_{t}^{\ast})\text{ on }\{Y_{t}^{\ast
}=V(t,X_{t}^{\ast})\} \cap \{p_{t}=\bar{p}_{t}\},\text{ }P\text{-a.s.,}%
\]
where $p_{t}$ and $\bar{p}_{t}$ are defined in (\ref{nn-e3-15}).
\end{corollary}

The following example shows that the conclusion obtained from Theorem
\ref{MP-DPP-2} is better than that from Theorem \ref{MP-DPP-1}.

\begin{example}
\label{exa1}Let $G(a):=\frac{1}{2}(\bar{\sigma}^{2}a^{+}-\underline{\sigma
}^{2}a^{-})$ for $a\in \mathbb{R}$ with $\bar{\sigma}\geq \underline{\sigma}>0$
and let $\tilde{G}:\mathbb{R}\rightarrow \mathbb{R}$ be a continuously
differentiable convex function dominated by $G$. Consider the following
control system:%
\[
\left \{
\begin{array}
[c]{l}%
dX_{t}^{u}=X_{t}^{u}d\langle B\rangle_{t}+u_{t}dB_{t},\text{ }X_{0}^{u}%
=x_{0}\in \mathbb{R},\\
Y_{t}^{u}=\mathbb{\tilde{E}}_{t}\left[  |X_{T}^{u}|^{2}-\int_{t}^{T}%
(2Y_{s}^{u}+u_{s})d\langle B\rangle_{s}\right]  ,
\end{array}
\right.
\]
where the control domain $U=[0,1]$. The related HJB equation is
\[
\partial_{t}V(t,x)+\inf_{v\in \lbrack0,1]}\tilde{G}(v^{2}\partial_{xx}%
^{2}V(t,x)+2x\partial_{x}V(t,x)-4V(t,x)-2v)=0,\text{ }V(T,x)=x^{2}.
\]
It is easy to verify that $V(T,x)=x^{2}+l(t)$, where $l(\cdot)$ satisfies the
following ordinary differential equation:%
\[
l^{\prime}(t)+\tilde{G}(-4l(t)-0.5)=0,\text{ }l(T)=0.
\]
By Theorem 5.10 in \cite{MJL0}, we know that $u_{t}^{\ast}=0.5$ for
$t\in \lbrack0,T]$ is an optimal control. It is easy to check that%
\[
Y_{t}^{\ast}=V(t,X_{t}^{\ast})=|X_{t}^{\ast}|^{2}+l(t),\text{ q.s.,}%
\]%
\[
p_{t}^{Q}=2X_{t}^{\ast},\text{ }q_{t}^{Q}=1,\text{ }N_{t}^{Q}=0\text{ for each
}Q\in \mathcal{P},
\]%
\[
\bar{K}_{T}=\frac{1}{2}\int_{0}^{T}(-4l(t)-0.5)d\langle B\rangle_{t}-\int
_{0}^{T}\tilde{G}(-4l(t)-0.5)dt,
\]
and%
\[
\xi^{\ast}=Y_{0}^{\ast}+\int_{0}^{T}X_{s}^{\ast}dB_{s}+\bar{K}_{T}.
\]
By Proposition 5.1 in \cite{LH}, we obtain that $\mathcal{\bar{P}}^{\ast
}=\mathcal{P}^{\ast}$ and $\mathcal{\bar{P}}^{\ast}$ contains only one
probability $P^{\ast}$, where%
\[
d\langle B\rangle_{t}=2\tilde{G}^{\prime}(-4l(t)-0.5)dt\text{ for }t\in
\lbrack0,T]\text{ under }P^{\ast}\text{.}%
\]
By Theorem \ref{MP-DPP-1}, we can get $\partial_{x}V(t,X_{t}^{\ast}%
)=2X_{t}^{\ast}$ under $P^{\ast}$. Note that $\mathcal{\tilde{P}}_{t}^{\ast
}=\mathcal{P}$ and $p_{t}=\bar{p}_{t}=2X_{t}^{\ast}$ under any $P\in
\mathcal{\tilde{P}}_{t}^{\ast}$. Then we obtain $\partial_{x}V(t,X_{t}^{\ast
})=2X_{t}^{\ast}$, q.s., by Theorem \ref{MP-DPP-2}.
\end{example}

\section*{Appendix}

\stepcounter{section}

For the convenience of the reader, we give the proof of (\ref{nnn-e3-4}) by
using the same method as in the proof of Theorem 4.4 in \cite{LH}.

\textbf{Proof of (\ref{nnn-e3-4}).} Set $\hat{Y}_{t}=Y_{t}^{\ast}-\bar{Y}_{t}%
$, $\tilde{g}_{y}(t)=\int_{0}^{1}g_{y}(t,X_{t}^{\ast},\bar{Y}_{t}+\alpha
\hat{Y}_{t},u_{t}^{\ast})d\alpha$, $\Lambda_{t}=\exp \left(  \int_{0}^{t}%
\tilde{g}_{y}(s)d\langle B\rangle_{s}\right)  $ for $t\in \lbrack0,T]$ and%
\[
\bar{\xi}=\Phi(X_{T}^{\ast})-\int_{0}^{T}\Sigma_{s}ds+\int_{0}^{T}%
g(s,X_{s}^{\ast},\bar{Y}_{s},u_{s}^{\ast})d\langle B\rangle_{s}.
\]
Then, by (\ref{nnn-e3-3}) and (\ref{n-nnn-e3-3}), we get%
\[
\hat{Y}_{t}\geq E_{P^{\ast}}\left[  \left.  \int_{t}^{T}\Sigma_{s}ds+\int
_{t}^{T}\tilde{g}_{y}(s)\hat{Y}_{s}d\langle B\rangle_{s}\right \vert
\mathcal{F}_{t}\right]  ,\text{ }P^{\ast}\text{-a.s.}%
\]
Set $\delta_{t}=\hat{Y}_{t}-E_{P^{\ast}}\left[  \left.  \int_{t}^{T}\Sigma
_{s}ds+\int_{t}^{T}\tilde{g}_{y}(s)\hat{Y}_{s}d\langle B\rangle_{s}\right \vert
\mathcal{F}_{t}\right]  $ for $t\in \lbrack0,T]$, it is easy to check that%
\begin{align*}
\delta_{t}  &  =\mathbb{\tilde{E}}_{t}[\xi^{\ast}]-\mathbb{\tilde{E}}_{t}%
[\bar{\xi}]-E_{P^{\ast}}[\xi^{\ast}-\bar{\xi}|\mathcal{F}_{t}]\\
&  =\mathbb{\tilde{E}}_{t}[\xi^{\ast}]-E_{P^{\ast}}[\xi^{\ast}|\mathcal{F}%
_{t}]+\alpha_{t}^{T}(P^{\ast}),\text{ }P^{\ast}\text{-a.s.}%
\end{align*}
Then we have%
\[
\hat{Y}_{t}-\delta_{t}=E_{P^{\ast}}\left[  \left.  \int_{t}^{T}\Sigma
_{s}ds+\int_{t}^{T}[\tilde{g}_{y}(s)(\hat{Y}_{s}-\delta_{s})+\tilde{g}%
_{y}(s)\delta_{s}]d\langle B\rangle_{s}\right \vert \mathcal{F}_{t}\right]  .
\]
From this we can deduce%
\begin{equation}
\hat{Y}_{0}-\delta_{0}=E_{P^{\ast}}\left[  \int_{0}^{T}\Sigma_{t}\Lambda
_{t}dt+\int_{0}^{T}\tilde{g}_{y}(t)\delta_{t}\Lambda_{t}d\langle B\rangle
_{t}\right]  . \label{ap-1}%
\end{equation}
Since $\Sigma_{t}\geq0$ and $\Lambda_{t}\geq C=\exp(-L\bar{\sigma}^{2}T)$ for
$t\in \lbrack0,T]$, we get%
\begin{equation}
\int_{0}^{T}\Sigma_{t}\Lambda_{t}dt\geq C\int_{0}^{T}\Sigma_{t}dt.
\label{ap-2}%
\end{equation}
In the following, we prove%
\begin{equation}
\delta_{0}+E_{P^{\ast}}\left[  \int_{0}^{T}\tilde{g}_{y}(t)\delta_{t}%
\Lambda_{t}d\langle B\rangle_{t}\right]  \geq C\left(  \mathbb{\tilde{E}}%
[\xi^{\ast}]-E_{P^{\ast}}[\xi^{\ast}]+\alpha_{0}^{T}(P^{\ast})\right)  .
\label{ap-3}%
\end{equation}
For each given $N\geq1$, set $\delta_{t}^{N}=\sum_{i=1}^{N}\delta_{t_{i}^{N}%
}I_{(t_{i-1}^{N},t_{i}^{N}]}(t)$, where $t_{i}^{N}=iTN^{-1}$ for $i\leq N$. By
Lemma 2.16 in \cite{HSW}, we have%
\begin{equation}
\lim_{N\rightarrow \infty}\sup_{|s-t|\leq TN^{-1}}E_{P^{\ast}}\left[
|\mathbb{\tilde{E}}_{s}[\xi^{\ast}]-\mathbb{\tilde{E}}_{t}[\xi^{\ast
}]|+|\mathbb{\tilde{E}}_{s}[\bar{\xi}]-\mathbb{\tilde{E}}_{t}[\bar{\xi
}]|\right]  =0. \label{ap-4}%
\end{equation}
Noting that (\ref{ap-4}) and $(E_{P^{\ast}}[\xi^{\ast}-\bar{\xi}%
|\mathcal{F}_{t}])_{t\leq T}$ is uniformly integrable under $P^{\ast}$, we can
deduce%
\[
\lim_{N\rightarrow \infty}E_{P^{\ast}}\left[  \int_{0}^{T}\tilde{g}%
_{y}(t)\delta_{t}^{N}\Lambda_{t}d\langle B\rangle_{t}\right]  =E_{P^{\ast}%
}\left[  \int_{0}^{T}\tilde{g}_{y}(t)\delta_{t}\Lambda_{t}d\langle
B\rangle_{t}\right]  .
\]
Note that $\int_{t_{i-1}^{N}}^{t_{i}^{N}}\tilde{g}_{y}(t)\Lambda_{t}d\langle
B\rangle_{t}=\Lambda_{t_{i}^{N}}-\Lambda_{t_{i-1}^{N}}$, $\delta_{T}=0$ and
$\Lambda_{0}=1$. Then we have%
\begin{align*}
\delta_{0}+E_{P^{\ast}}\left[  \int_{0}^{T}\tilde{g}_{y}(t)\delta_{t}%
^{N}\Lambda_{t}d\langle B\rangle_{t}\right]   &  =\delta_{0}+E_{P^{\ast}%
}\left[  \sum_{i=1}^{N}\delta_{t_{i}^{N}}(\Lambda_{t_{i}^{N}}-\Lambda
_{t_{i-1}^{N}})\right] \\
&  =\sum_{i=0}^{N-1}E_{P^{\ast}}\left[  \Lambda_{t_{i}^{N}}(\delta_{t_{i}^{N}%
}-E_{P^{\ast}}[\delta_{t_{i+1}^{N}}|\mathcal{F}_{t_{i}^{N}}])\right]  .
\end{align*}
Since $\delta_{t}=\mathbb{\tilde{E}}_{t}[\xi^{\ast}]-E_{P^{\ast}}[\xi^{\ast
}|\mathcal{F}_{t}]+\alpha_{t}^{T}(P^{\ast}),$ $P^{\ast}$-a.s., by Remark 3.6
in \cite{LH}, we get $P^{\ast}$-a.s.%
\[
\delta_{t_{i}^{N}}-E_{P^{\ast}}[\delta_{t_{i+1}^{N}}|\mathcal{F}_{t_{i}^{N}%
}]=\alpha_{t_{i}^{N}}^{t_{i+1}^{N}}(P^{\ast})-\left(  E_{P^{\ast}%
}[\mathbb{\tilde{E}}_{t_{i+1}^{N}}[\xi^{\ast}]|\mathcal{F}_{t_{i}^{N}%
}]-\mathbb{\tilde{E}}_{t_{i}^{N}}[\mathbb{\tilde{E}}_{t_{i+1}^{N}}[\xi^{\ast
}]]\right)  \geq0,
\]
where%
\[
\alpha_{t_{i}^{N}}^{t_{i+1}^{N}}(P^{\ast})=\underset{\xi \in L_{G}^{1}%
(\Omega_{t_{i+1}^{N}})}{ess\sup}^{P^{\ast}}\left(  E_{P^{\ast}}[\xi
|\mathcal{F}_{t_{i}^{N}}]-\mathbb{\tilde{E}}_{t_{i}^{N}}[\xi]\right)  ,\text{
}P^{\ast}\text{-a.s.}%
\]
Thus we get%
\begin{align*}
\sum_{i=0}^{N-1}E_{P^{\ast}}\left[  \Lambda_{t_{i}^{N}}(\delta_{t_{i}^{N}%
}-E_{P^{\ast}}[\delta_{t_{i+1}^{N}}|\mathcal{F}_{t_{i}^{N}}])\right]   &  \geq
C\sum_{i=0}^{N-1}E_{P^{\ast}}\left[  \delta_{t_{i}^{N}}-E_{P^{\ast}}%
[\delta_{t_{i+1}^{N}}|\mathcal{F}_{t_{i}^{N}}]\right] \\
&  =C\sum_{i=0}^{N-1}E_{P^{\ast}}\left[  \delta_{t_{i}^{N}}-\delta
_{t_{i+1}^{N}}\right] \\
&  =C\delta_{0},
\end{align*}
which implies (\ref{ap-3}). By (\ref{ap-1})-(\ref{ap-3}), we obtain
(\ref{nnn-e3-4}). $\Box$

\end{document}